\documentclass[10pt,journal,compsoc]{IEEEtran}

\usepackage[utf8]{inputenc} 
\usepackage[T1]{fontenc}    
\usepackage{hyperref}       
\usepackage{url}            
\usepackage{booktabs}       
\usepackage{amsfonts}       
\usepackage{nicefrac}       
\usepackage{microtype}      
\usepackage{graphicx,color}
\usepackage{latexsym,amsmath,amssymb,amsfonts}
\usepackage{comment}
\usepackage{algorithm}
\usepackage[noend]{algpseudocode}
\usepackage{amsthm}
\usepackage{subcaption}
\usepackage{mathabx}
\usepackage{multirow}

\algblockdefx{FORALLP}{ENDFAP}[1]%
  {\textbf{for all }#1 \textbf{do in parallel}}%
  {\textbf{end for}}
  
\hypersetup{colorlinks,
            linkcolor=blue,
            citecolor=blue,
            urlcolor=magenta,
            linktocpage,
            plainpages=false}

\newtheorem{thm}{Theorem}[section]
\newtheorem{lem}{Lemma}[section]

\newtheorem{cor}{Corollary}[section]
\newtheorem{cond}{Condition}[section]
\newtheorem*{remark}{Remark}

\ifCLASSOPTIONcompsoc
  \usepackage[nocompress]{cite}
\else
  \usepackage{cite}
\fi
\ifCLASSINFOpdf
\else
\fi
\begin{document}
%
\title{Adaptive Variational Bayesian Inference for Sparse Deep Neural Network}
%
%
%
%

\author{Jincheng~Bai,
        Qifan~Song,
        and~Guang~Cheng
\IEEEcompsocitemizethanks{\IEEEcompsocthanksitem The authors are with the Department
of Statistics, Purdue University, West Lafayette,
IN, 47906.\protect\\
E-mail: \{bai45, qfsong, chengg\}@purdue.edu}
}

\IEEEtitleabstractindextext{%
\begin{abstract}
In this work, we focus on variational Bayesian inference on the sparse Deep Neural Network (DNN) modeled under a class of spike-and-slab priors. Given a pre-specified sparse DNN structure, the corresponding variational posterior contraction rate is characterized that reveals a trade-off between the variational error and the approximation error, which are both determined by the network structural complexity (i.e., depth, width and sparsity). However, the optimal network structure, which strikes the balance of the aforementioned trade-off and yields the best rate, is generally unknown in reality. Therefore, our work further develops an {\em adaptive} variational inference procedure that can automatically select a reasonably good (data-dependent) network structure that achieves the best contraction rate, without knowing the optimal network structure. In particular, when the true function is H{\"o}lder smooth, the adaptive variational inference is capable to attain (near-)optimal rate without the knowledge of smoothness level. The above rate still suffers from the curse of dimensionality, and thus motivates the teacher-student setup, i.e., the true function is a sparse DNN model, under which the rate only logarithmically depends on the input dimension. 
\end{abstract}

}

\maketitle

\IEEEdisplaynontitleabstractindextext

%
\IEEEpeerreviewmaketitle

\section{Introduction}

Deep Neural Networks (DNNs) have achieved tremendous successes in AI fields such as computer vision, natural language processing and reinforcement learning. 
One crucial factor for the successes of DNN is that it possesses highly complex and nonlinear model architecture, which allows it to approximate almost any complicated function \cite{Cybenko1989Approximation,Rolnick2018power,Mhasker2017when}.

However, large and deep fully connected networks are memory demanding \cite{Srivastava15dropout} and also slow in inference for some real time tasks, which sheds the light in the use of  sparse neural nets. Meanwhile, sparse neural nets have been shown to have accurate approximation and strong generalization power \cite{Glorot2011Deep, Goodfellow2016Deep}. For example, the popular Dropout regularization \cite{Srivastava15dropout} could be interpreted as averaging over $l_0$ regularized sparse neural nets. From a nonparametric perspective, \cite{Schmidt-Hieber2017Nonparametric} showed that sparse DNN with a ReLU activation function could achieve nearly minimax rate in the regression setup.

Bayesian neural nets (BNN) are perceived to perform well against overfitting due to its regularization nature by enforcing a prior distribution. The study of Bayesian neural nets could date back to \cite{Mackay1992practical}, \cite{Neal1993Bayesian}. In particular, a spike-and-slab prior \cite{George1993variable} would switch a certain neuron off, and thus in nature imposes $l_0$ regularization and encourages network sparsity. \cite{Polson2018posterior} introduced the Spike-and-Slab Deep Learning as a fully Bayesian alternative to Dropout for improving the generalizability of DNN with ReLU activation, where the posterior distribution is proven to concentrate at a nearly minimax rate.  

However, a well-known obstacle for Bayesian inference is its high computational cost for drawing samples from posterior distribution via Markov chain Monte Carlo (MCMC). A popular alternative - Variational Inference (VI) or Variational Bayes (VB) \cite{Jordan1999Variational} - approximates the true posterior distribution by a simpler family of distributions through an optimization over the Evidence Lower Bound (ELBO). 
As a computationally efficient method, VI has been used widely for neural networks \cite{graves2011practical,Kingma2014VAE, Rezende2014Stochastic, Blundell2015weight}.
However, statistical properties of VI have not been carefully studied only until recently \cite{Pati2018on, Alquier2017Concentration, Wang2019Frequentist}, and the convergence property for variational BNN remains much less explored. Specifically, it would be interesting to examine whether the variational inference leads to the same rate of convergence
compared to the Bayesian posterior distribution and frequentist estimators. \cite{Cherief2019Convergence}
attempts to provide theoretical justifications for variational inference on BNN but only for an inflated {\em tempered posterior} \cite{Bhattacharya2019Bayesian} rather than the true posterior.

In this paper, we directly investigate the theoretical behavior of variational posterior for Bayesian DNN under spike-and-slab modeling. 
Our specific goals are to understand how fast the variational posterior converges to the truth and how accurate the prediction carried out by variational inferences is. It is not surprising that the choice of the network structure, i.e., network depth, width and sparsity level, plays a crucial role for the success of variational inference. Notably, there exists a trade-off phenomenon for the choice of network architecture: an overly complex structure leads to a large variational error, while an overly simplified network may not be able to capture the nonlinear feature of true underlying regression function (i.e., large approximation error).


The optimal network structure, which yields the best contraction rate, is generally unknown in reality. This motivates us to develop an {\em adaptive} variational inference procedure that performs automatic variational architecture selection based on the penalized ELBO criterion. The selection procedure could lead to a data-dependent network structure that achieves the same best rate as if it were derived under the optimal structure choice.

The developed general theory is further applied to two particular examples, where the true underlying function 1) is H{\"o}lder smooth, or 2) exactly corresponds to some unknown sparse DNN model. For the formal case, we show that if the smoothness level is known, the variational posterior possesses minimax contraction rate (up to a logarithm factor) when the network structure is carefully chosen based on the known smoothness level. Even when the smoothness level is unknown, the proposed adaptive variational inference procedure still leads to the same theoretical guarantee. For the latter case, we find that the rate of convergence doesn't suffer from the curse of dimensionality, in the sense that the input dimension has at most a logarithmic effect to the convergence rate.

It is worth noting that the focus of this paper lies on the theory of variational inference on sparse DNN, and the prior used for deriving the theoretical results leads to intractable ELBO optimization.
Although the variational inferences could be implemented by utilizing certain approximation, as illustrated in the supplementary material, computation-friendly priors will be developed in the future work.


\subsection{Related work}
There exists a rich literiture of sparsifying DNN based on ``train-and-prune" strategy 
\cite{Blundell2015weight, Molchanov2017Variational, Louizos2017Bayesian, ghosh2017Model, Gale2019state}. This class of approaches first train a fully connected (usually over-parameterized) DNN, and then attempt to sparsify it by pruning connections with “small” weights. The pruning could be either based on absolute magnitude of the weights, or based on variational distribution of the weights. Comparing to our method that directly induces sparsity, the “train-and-prune” strategy usually requires at least an additional round of training after pruning; and to the best of our knowledge, there is no theoretical justification available for such “train-and-prune” approaches yet. Moreover, we would like to mention that there is no (theoretically guaranteed) adaptive way to determine the optimal pruning rate along this line of works. In contrast, our approach is one shot and can be incorporated with adaptive priors to automatically choose the optimal sparsity level.

\subsection{Notations}
Throughout this paper, the following notations are used.
Denote $\mbox{KL}(\cdot\|\cdot)$ and $d(\cdot, \cdot)$ as the KL divergence and Hellinger distance between two probability measures, respectively. For a vector $\boldsymbol{x}=(x_1, \ldots, x_m)^T$, we define $\|\boldsymbol{x}\|_{\infty}:=\max^m_{i=1}|x_i|$, $\|\boldsymbol{x}\|_0:=\sum^m_{i=1}I(x_i \neq 0)$, $\|\boldsymbol{x}\|_p:=(\sum^m_{i=1}|x_i|^p)^{1/p}$ for $p>0$. For any Lebesgue integrable function $f$, we denote the $L_p$ norm for $f$ as $\|f\|_p:=(\int f^p)^{1/p}$ and $\|f\|_{\infty}:=\sup_{y \in \mathcal{Y}}|f(y)|$.

\section{Nonparametric Regression Via Bayesian Deep Learning}{\label{secNonparam}}
Consider a nonparametric regression model with random covariates $X_i \sim \mathcal{U}([-1,1]^p)$\footnote{The bounded support assumption is common in the literature (\cite{Schmidt-Hieber2017Nonparametric, Polson2018posterior}) and applies to standardized data.} and 
\begin{equation}{\label{eqmod1}}
   Y_i = f_0(X_i)+\epsilon_i, \: i=1,\ldots,n 
\end{equation}
where $\mathcal{U}$ denotes the uniform distribution, $\epsilon_i\overset{iid}{\sim}\mbox{N}(0,\sigma^2_{\epsilon})$ is the noise term, and $f_0:[-1,1]^p\rightarrow \mathbb{R}$ is the underlying true function. For simplicity of the analysis, we assume that $\sigma_{\epsilon}$ is a known constant, while in practice we could use the empirical Bayes method or full Bayes method (by placing an Inverse-Gamma prior on $\sigma_{\epsilon}$) to estimate it.

\subsection{Deep neural networks}

An ($L$-1)-hidden-layer ReLU neural network is used to model the data. 
The number of neurons in each hidden layer is denoted by $p_i$ for $i=1,\dots, L-1$. The weight matrix and bias parameters in each layer are denoted by $W_i\in\mathbb{R}^{p_{i-1}\times p_{i}}$ and $b_i\in\mathbb{R}^{p_{i}}$ for $i=1,\dots,L$.
Let $\sigma(x)=\max(0,x)$ be the ReLU activation function, and for any $r\in\mathbb{Z}^+$ and any $b\in\mathbb{R}^r$, we define $\sigma_b: \mathbb{R}^r \rightarrow \mathbb{R}^r$ as
$\sigma_b(y_i)=\sigma(y_i - b_i)$, for $i=1,\ldots, r$.
Therefore, given parameters $\boldsymbol{p}=(p_1,\dots,p_{L-1})'$ and $\theta = \{W_1, b_1,\ldots, W_{L}, b_{L}\}$,
the output of this DNN model can be written as 
\begin{equation}{\label{eqmod2}}
    f_{\theta}(X)=W_{L}\sigma_{b_{L}}(W_{L}\sigma_{b_{L-1}}\ldots \sigma_{b_1}(W_1X))+b_L.
\end{equation}
In what follows, with slight abuse of notation, $\theta$ is also viewed as a vector that contains all the coefficients in $W$'s and $b$'s, and its length is denoted by $H$, i.e., $\theta=(\theta_1,\dots,\theta_H)'$.


\subsection{Regularization via spike-and-slab prior}
Instead of using a fully connected neural net, i.e., $\theta$ is a dense vector, we consider a sparse NN $f_{\theta} \in \mathcal{F}(L,\boldsymbol{p},s)$, where
$$
\mathcal{F}(L,\boldsymbol{p},s)=\{f_{\theta} \mbox{  as in (\ref{eqmod2})}: \|\theta\|_{0}\leq s \},$$
$s \in \mathbb{N}$ controls the sparsity level of NN connectivity. The set of $\theta$ under the constraint $\mathcal{F}(L,\boldsymbol{p},s)$ is denoted as $\Theta(L,\boldsymbol{p},s)$. 


Given a specified sparse network configuration, we impose a fully Bayesian modeling with a spike-and-slab prior on $\theta$. Denoting $\delta_{0}$ as the Dirac at 0 and $\gamma = (\gamma_1, \ldots, \gamma_H)$ as a binary vector indicating the inclusion of each edge in the network. The prior distribution $\pi(\theta)$ thus follows:
\begin{equation}\label{prior}
\begin{split}
\pi(\theta_i|\gamma_i) = \gamma_i\mathcal{M}_0(\theta_i)+(1-\gamma_i)\delta_{0},
\: \pi(\gamma)\propto 1\{\sum \gamma_i=s\}
\end{split}
\end{equation}
for $1\leq i\leq H$, where we assign uniform prior over all possible $s$-sparse network structures, and the slab distribution
$\mathcal{M}_0(\theta_i)$ is either a uniform  distribution $\mathcal{U}([-B_0,B_0])$ or a Gaussian distribution $\mathcal{N}(0, \sigma^2_0)$ with predetermined constant $B_0>1$ and $\sigma^2_0>0$.
Our developed theory holds for both uniform slab and Gaussian slab modeling. 

We denote $D_i=(X_i, Y_i)$ and $\boldsymbol{D}^{(n)}=(D_1,\dots, D_n)$ as the observations.
Let $P_0$ denote the underlying probability measure of data, and $p_0$ denote the corresponding density function, i.e., $p_0(D_i)=\phi([Y_i-f_0(X_i)]/\sigma_{\varepsilon})/\sigma_{\varepsilon}$ where $\phi$ is the normal pdf. Similarly, let $P_\theta$ and $p_\theta$ be the distribution and density functions induced by the parametric NN model ($\ref{eqmod2}$).
Thus, the posterior distribution is written as
$\pi(\theta|\boldsymbol{D}^{(n)})\propto \pi(\theta) \cdot p_{\theta}(\boldsymbol{D}^{(n)}).$

\section{Variational Inference}
In the framework of variational inference, one seeks to find a good approximation of the posterior $\pi(\theta|\boldsymbol{D}^{(n)})$ via optimization rather than to simulate the posterior distribution by long-run Markov chain Monte Carlo. Given a variational family of distributions, denoted by $\mathcal Q$, the goal is to minimize the KL divergence between distributions in $\mathcal Q$ and true posterior distribution: 
\begin{equation}\label{eqvb1}
\widehat{q}(\theta) = \arg\min_{q(\theta) \in \mathcal{Q}}\mbox{KL}(q(\theta)\|\pi(\theta|\boldsymbol{D}^{(n)})),
\end{equation}
and the variational posterior $\widehat q(\theta)$ is subsequently used for approximated inference.
To solve the optimization problem ($\ref{eqvb1}$), we note that 
$\mbox{KL}(q(\theta)\|\pi(\theta|\boldsymbol{D}^{(n)}))=C-\Omega$, where $C$ is some constant depending on data $\boldsymbol{D}^{(n)}$ only, and $$\Omega:=\mathbb{E}_{q(\theta)}[\log\frac{p_\theta(\boldsymbol{D}^{(n)})\pi(\theta)}{q(\theta)}]$$ is the so-called Evidence Lower Bound (ELBO). Then an equivalent optimization to ($\ref{eqvb1}$) is $$
\widehat{q}(\theta) = \arg\max_{q(\theta) \in \mathcal{Q}}\Omega,
$$ which is usually conducted via gradient ascent type algorithms.


An inspiring representation of $\Omega$ is 
\begin{equation}\label{eqvb3}
\begin{split}
-\Omega = -\mathbb{E}_{q(\theta)}[\log p_{\theta}(\boldsymbol{D}^{(n)})] + \mbox{KL}(q(\theta)\|\pi(\theta)),
\end{split}
\end{equation}
where the first term in ($\ref{eqvb3}$) can be viewed as the reconstruction error \cite{Kingma2014VAE} and the second term serves as regularization. Hence the variational inference procedure tends to be minimizing the reconstruction error while being penalized against prior distribution in the sense of KL divergence.

Technically, the variational family $\mathcal{Q}$ can be chosen freely. But for the sake of efficient implementation and optimization, it is often selected as some simple distribution family. In our case, $\mathcal{Q}$ is chosen as the spike-and-slab distribution to resemble the prior distribution, i.e., for $i=1,\dots, H$,
\begin{equation}\label{vbpost}
   q(\theta_i|\gamma_i) = \gamma_i\mathcal{M}(\theta_i)+(1-\gamma_i)\delta_{0},\:
   q(\gamma_i) =\mbox{Bern}(\nu_i),
\end{equation}
where $\mathcal{M}(\theta_i)$ is either  $\mathcal{U}(l_i, u_i)$ with $-B_0 \leq l_i\leq u_i \leq B_0$ or $\mathcal{N}(\mu_i, \sigma^2_i)$ depending on the slab choice $\mathcal M_0$ in (\ref{prior}), and $0\leq \nu_i \leq 1$. Note that since the posterior can not have a larger support than the prior distribution, the ELBO optimizer must satisfy $\widehat\nu_i\in\{0,1\}$ and $\sum \widehat\nu_i =s$.

\section{VB Posterior Asymptotics} \label{sec:rate}
In this section, we establish the distributional convergence of the variational Bayes posterior $\widehat q(\theta)$, towards the true regression function $f_0$, under the squared Hellinger distance $d(\cdot,\cdot)$, which is
\[
d^2(P_\theta,P_0)=\mathbb{E}_X\left(1-\exp\left\{-\frac{[f_\theta(X)-f_0(X)]^2}{8\sigma^2_{\epsilon}}\right\}\right).
\]
Note that in section \ref{sec:gen}, the results under $L_2$ norm will be studied.

Denote the log-likelihood ratio between $p_0$ and $p_\theta$ as
$$
l_n(P_0,P_{\theta})=\log\frac{p_0(\boldsymbol{D}^{(n)})}{p_{\theta}(\boldsymbol{D}^{(n)})}=\sum^n_{i=1}\log \frac{p_0(D_i)}{p_{\theta}(D_i)},
$$
then the negative ELBO can be expressed as 
\[
-\Omega=\mbox{KL}(q(\theta)\|\pi(\theta))+\int l_n(P_0, P_{\theta})q(\theta)d\theta+C,
\]
where $C=-\log p_0(\boldsymbol{D}^{(n)})$ is a constant with respect to $q(\theta)$.

Our first lemma provides an upper bound for the negative ELBO for sparse DNN model under the prior specification (\ref{prior}) and variational family $\mathcal Q$. 
Let $\Theta_B(L,\boldsymbol{p},s)=\{\theta \in \Theta(L, \boldsymbol{p}, s):\|\theta\|_{\infty} \leq B\}$ for some constant $B>0$.
\begin{lem}{\label{thmbound2}} 
Given any network family $\mathcal F(L,\boldsymbol{p},s)$ with an equal width $\boldsymbol{p}=(12pN,\dots,12pN)$, we have that, with dominating probability for some $C' > 0$, 
\begin{align}
\inf_{q(\theta) \in \mathcal Q}\Bigl\{ \mbox{KL}(q(\theta)\|\pi(\theta))
+ \int l_n(P_0, P_{\theta})&q(\theta)d\theta \Bigr\}\nonumber \\
&\: \leq C'n (r_n
+ \xi_n ) {\label{eqbound2}}
\end{align}
holds, where $$r_n:=r_n(L, N, s)=\frac{Ls}{n}\log(12BpN) + \frac{s}{n}\log(nL/s),$$ and $$\xi_n := \xi_n(L, N, s)= \inf_{\theta \in \Theta_B(L, \boldsymbol{p}, s)}\|f_{\theta}-f_0\|^2_\infty,$$
where $B=B_0$ under uniform prior setting, and $B\geq2$ under normal prior setting. 
\end{lem}

The upper bound (\ref{eqbound2}) consists of two terms: the first term $r_n$ is the variational error caused by the variational Bayes approximation; the second term $\xi_n$ is the approximation error of approximating $f_0$ by sparse ReLU DNN whose weight and bias parameters are bounded by $B$. Note that since $B$ is a pre-specific constant, its value doesn't affect the rate of $r_n$

Our next lemma links the contraction rate of variational posterior with the negative ELBO discussed in Lemma \ref{thmbound2}.
\begin{lem}{\label{lmbound1}}
Given network family $\mathcal F(L,\boldsymbol{p},s)$ with equal width $\boldsymbol{p}=(12pN,\dots,12pN)$, if
$\max\{s\log(nL/s), Ls\log(pN)\} = o(n)$, then with probability at least $(1-e^{-Cn\varepsilon^2_n})$ for some $C>0$, we have
\begin{align}
\int d^2(P_{\theta}, P_0)\widehat{q}(\theta)d\theta \leq C\varepsilon^2_n + &\frac{3}{n}\inf_{q(\theta) \in \mathcal Q}\Bigl\{ \mbox{KL}(q(\theta)\|\pi(\theta))\nonumber\\
&\quad + \int l_n(P_0, P_{\theta})q(\theta)d\theta \Bigr\}, \label{eqbound1}
\end{align}
where $$\varepsilon_n:=\varepsilon_n(L,N,s)
=M\sqrt{\frac{s\log(nL/s) + Ls\log(pN)}{n}}\log^\delta(n)$$ for any $\delta\geq1$ and some large constant $M$.
\end{lem}

Note that Lemma \ref{lmbound1} holds regardless of the choice of prior specification $\pi(\theta)$ and variational family $\mathcal Q$.

The LHS of (\ref{eqbound1}) is the variational Bayes posterior mean of the squared Hellinger distance. On the RHS, the first term $\varepsilon_n$ represents the estimation error under Hellinger metric, such that it is possible to test the true distribution $P_0$ versus all alternatives $\{P_\theta: d(P_\theta, P_0)\geq \varepsilon_n, \theta\in\Theta(L,\boldsymbol{p},s)\}$ with
exponentially small error probability (refer to Lemma 1.2 in the supplementary material); the second term, as discussed above, is the negative ELBO (up to a constant), which has been elaborated in Lemma \ref{thmbound2}.

Combining the above two lemmas together, one can easily obtain the following theorem: 
\begin{thm}\label{thm:d2bound}
Given any network family $\mathcal F(L,\boldsymbol{p},s)$ with equal width $\boldsymbol{p}=(12pN,\dots,12pN)$, if the conditions of Lemmas \ref{thmbound2} and \ref{lmbound1} hold, then
\begin{equation}\label{d2bound}
\int d^2(P_{\theta}, P_0)\widehat{q}(\theta)d\theta \leq C\varepsilon^2_n + 3 C' r_n + 3 C'\xi_n.
\end{equation}
\end{thm}
The three terms in the RHS of (\ref{d2bound}) correspond to estimation error, variational error and approximation error respectively. All the three terms depend on the complexity of network structure. Specifically,
$$
\varepsilon_n^2\sim r_n\sim\max\left(\frac{s\log(nL/s)}{n},\frac{Ls\log(pN)}{n}\right),
$$
up to only logarithmic difference. Thus both $\varepsilon_n^2$ and $r_n$ are nearly linearly dependent on the sparsity and depth of the network structure specification. On the other hand, the approximation error $\xi_n$ generally decreases as one increases the complexity of networks configuration (i.e., the values of $N$, $L$ and $s$). 
Therefore, it reveals a trade-off phenomenon on the choice of network structure. Note that such trade-off echoes with those observed in the literature of nonparametric statistics: as one increases the domain of parameter space (e.g., increases the number of basis functions in spline regression modeling), it usually leads to smaller bias but larger variance.

As mentioned in \cite{Cherief2019Convergence}, we would like to bring out the concept of the bias-variance trade-off in the variational inference, where we name the third and second term in RHS of (\ref{eqbound2}) by bias and variance respectively. The variance component is controlled by $r_n$ with an order that is always linearly dependent on the sparsity level of the DNN, which is consistent with our perception. However, its linear dependence on the depth $L$ versus the logarithmic dependence on the width $N$ conflicts with the result that a deeper neural net generalizes better than a shallower one as often empirically observed. In the meantime, a deeper neural net could yield a smaller approximation error with fixed neurons \cite{Rolnick2018power}, which would then compensate for the increased variance caused by a deeper neural net. This reveals an interesting bias-variance trade-off phenomenon.

\section{Adaptive Architecture Search} \label{sec:adapt}
In Section \ref{sec:rate}, we establish the distributional convergence of VB posterior (\ref{d2bound}) under the Hellinger metric,
with a pre-specified DNN architecture, say depth $L$, width $N$ and sparsity $s$. 
Ideally, one would like to choose the network structure that minimizes the RHS of (\ref{d2bound}), thus leading to a better convergence guarantee. However, this best choice is generally not available due to the fact that the approximation error $\xi_n$ critically depends on the nature (e.g., continuity and smoothness) of the unknown $f_0$. Therefore, in this section, we will develop an adaptive variational Bayes inference procedure, under which the variational posterior contraction achieves the same convergence rate as if the optimal choice of network structure was given.

To simplify our analysis, we assume that the network depth $L$ is already well specified, and are only concerned about the adaptivity with respect to the network width and sparsity. Note that for a certain family of $f_0$, e.g., $f_0$ is H{\"o}lder smooth, the optimal choice of $L$ can indeed be specified without additional knowledge of $f_0$ (refer to Section \ref{sec:app} for detail).
To be more specific, we define $$(N^*,s^*)=\arg\min_{N, s} \{r_n(s,L,N) + \xi_n(s, L, N)\},$$ and consider
$12pN^*$ and $s^*$ to be the optimal network structure configuration for width and sparsity respectively. Such a choice strikes an optimal balance between variational error and approximation error. It is worth mentioning that the estimation error term $\varepsilon_n^2$ is of the same order as $r_n$ (up to a logarithmic term). Therefore, the optimal choice $(s^*,N^*)$ does minimize the RHS of (\ref{d2bound}) (up to a logarithmic term).
We further define $$\varepsilon_n^*
=M'\sqrt{\frac{Ls^*\log N^*+s^*\log(Ln/s^*)}{n}}\log^\delta(n)
$$ for some constant $M'$, $r_n^*=r_n(L,N^*,s^*)$ and $\xi_n^*=\xi_n(L,N^*,s^*)$. They represent the estimation error, variational error and approximation error respectively, under optimal choices $N^*$ and $s^*$.

In addition, the following  conditions are imposed on the optimal values $N^*$ and $s^*$:
\begin{cond} \label{cond1}
$1\prec \max\{Ls^* \log(pN^*), s^*\log(nL/s^*)\} = o(n^\alpha)$ for some $\alpha<1$.
\end{cond}
\begin{cond} \label{cond2}
$r_n^* \asymp \xi_n^*$.
\end{cond}
\begin{cond} \label{cond3}
$s^*\geq 12p N^*+L$.
\end{cond}

Condition \ref{cond1} assumes that the optimal network structure, in the asymptotic sense, is a sparse one. This is reasonable as it essentially requires that the data can be well approximated by a sparse DNN model. If this condition fails, there will be no basis for conducting sparse DNN modeling. 
Condition \ref{cond2} implies that the choice $(N^*,s^*)$, which minimizes $r_n+\xi_n$, also strikes the balance between $r_n$ and $\xi_n$.
Condition \ref{cond3} avoids the redundancy of network width. If this condition is violated, then there must be redundant node (i.e., node without connection) in every hidden layers. In such a situation, all these redundant nodes shall be removed from the network configuration, leading to a narrower network.
 
In the Bayesian paradigm, the adaptivity can be achieved by impose a reasonable prior on $(N,s)$. In other words, we expand the prior support to $$\mathcal{F}=\bigcup_{N=1}^\infty\bigcup_{s=0}^{H_{N}}\mathcal{F}(L,\boldsymbol{p}_N^{L},s),$$ where $\boldsymbol{p}_N^{L}=(12pN,\dots,12pN)\in\mathbb{R}^{L}$ and $H_{N}$ is the total possible number of edges in the ($L$-1)-hidden-layer network with layer width $12pN$. The prior specification on the network structure is similar to  \cite{Polson2018posterior}, that is 
\begin{equation}\label{prior2}
\begin{split}
    &\pi(N)=\frac{\lambda^N}{(e^{\lambda}-1)N!}\quad \mbox{ for }N\geq1,\\
    &\pi(s)\propto e^{-\lambda_ss}\quad \mbox{ for }s\geq0,
\end{split}
\end{equation}
where $\lambda_s$ satisfies $n\varepsilon_n^{*2}/s^*\succ \lambda_s \geq aL\log n$ for some $a>0$.

To implement variational inference, we consider the variational family $\mathcal{Q}_{N,s}$ that restricts the VB marginal posterior of $N$ and $s$ to be a degenerate measure: every distribution $q(\theta, N, s)$ in $\mathcal{Q}_{N,s}$ follows
\begin{equation}\label{qfamily}
\begin{split}
    &q(N)=\delta_{\widebar N},\quad q(s)=\delta_{\widebar s}, \quad q(\gamma_i|N, s) = \mbox{Bern}(\nu_i),\\
    & q(\theta_i|\gamma_i) = \gamma_i\mathcal{M}(\theta_i)+(1-\gamma_i)\delta_{0},
\end{split}
\end{equation}
for some $\widebar N\in\mathbb{Z^+}$ and $\widebar s\in\mathbb{Z}^{\geq 0}$.
This choice of variational family means that the VB posterior will adaptively select one particular network structure $(\widehat N,\hat s)$ by minimizing
\[
\widehat q(\theta, N, s)=\underset{q(\theta, N, s)\in\mathcal Q_{N,s}}{\arg\max} 
\mbox{KL}(q(\theta, N, s)\|\pi(\theta,N,s|\boldsymbol{D}^{(n)})).
\]
Note that 
$
\mbox{KL}(q(\theta, N, s)\|\pi(\theta,N,s|\boldsymbol{D}^{(n)}))
=-\log\pi(\widebar N,\widebar s)+\mbox{KL}(q(\theta|\widebar N, \widebar s)\|p(\theta,\boldsymbol{D}^{(n)}|\widebar N,\widebar s))+C,
$
for some constant $C$. Let $$\Omega(\widebar N, \widebar s)=\max_{q(\theta|\widebar N, \widebar s)}[-\mbox{KL}(q(\theta| \widebar N, \widebar s)\|p(\theta,\boldsymbol{D}^{(n)}|\widebar N, \widebar s))]$$ 
be the maximized ELBO given the network structure determined by parameters $\widebar N$ and $\widebar s$.
Then 
\begin{equation} \label{eq:pelbo}
(\widehat N,\widehat s)=\arg\max_{\widebar N,\widebar s}[ \Omega(\widebar N, \widebar s)+\log\pi(\widebar N,\widebar s)].
\end{equation}
In other words, the above VB modeling leads to a variational network structure selection based on a penalized ELBO criterion, where the penalty term is the logarithm of the prior of $\widebar N$ and $\widebar s$.

In Bayesian analysis, model selection relies on the (log-)posterior: 
$\log\pi(D|\widebar N, \widebar s)+\log \pi(\widebar N,\widebar s)$. Thus, the proposed variational structure selection procedure is an approximation to maximum a posteriori (MAP) estimator, by replacing the model evidence term $\log\pi(D|\widebar N,\widebar s)$ with the ELBO $\Omega(\widebar N,\widebar s)$.


Our next theorem shows that the proposed variational modeling attains the best rate of convergence without the knowledge of optimal network architecture $N^*$ and $s^*$.
\begin{thm}\label{adapt}
Under the adaptive variational Bayes modeling described above, we achieve that
\begin{equation}\label{d2bound+}
\int d^2(P_{\theta}, P_0)\widehat{q}(\theta)d\theta \leq C''[\varepsilon_n^{*2} +r_n^* +\xi_n^*] 
\end{equation}
holds with dominating probability for some constant $C''>0$.
\end{thm}

It is worth mentioning that the above result doesn't imply the adaptive variational procedure exactly finds the optimal choice such that $\widehat N\approx N^*$ and $\widehat s\approx s^*$. The proof of Theorem \ref{adapt} only shows that the adaptive VB procedure avoids over-complicated network structures, such that $\widehat N$ and $\widehat s$ will not be overwhelmingly larger than the $N^*$ and $s^*$ respectively. Note that $(N^*,s^*)$ is the universal optimal choice, in the sense that it ensures that for any data set generated from the underlying model (\ref{eqmod1}), the corresponding variational inference is the best. Note that $(\widehat N, \widehat s)$ is a data-dependent choice, which differs from data to data and may be quite different from $(N^*,s^*)$.


\section{Applications} \label{sec:app}

In this section, we will apply the general theoretical results to two important types of ground truth: 1) $f_0$ is some unknown H{\"o}lder smooth function and 2) $f_0$ exactly corresponds to an unknown sparse DNN model, i.e., the teacher-student framework \cite{tian2018theoretical, Goldt2019Dynamics}.

\subsection{H{\"o}lder smooth function} \label{sec:knownalpha}
we assume the unknown $f_0$ belongs to the class of $\alpha$-H{\"o}lder smooth functions $\mathcal{H}^{\alpha}_p$, defined as 
\begin{equation*}
\begin{split}
\mathcal{H}^{\alpha}_p = \Bigl\{&f: \|f\|^\alpha_{\mathcal{H}}:= \sum_{\kappa:|\kappa|<\alpha}\|\partial^{\kappa}f\|_{\infty} \\
&+ \sum_{\kappa:|\kappa|=\lfloor \alpha \rfloor}\sup_{\substack{x,y\in[-1,1]^p\\ x \neq y}} \frac{|\partial^{\kappa}f(x)-\partial^{\kappa}f(y)|}{|x-y|_{\infty}^{\alpha - \lfloor \alpha \rfloor}}\leq \infty \Bigr\}.
\end{split}
\end{equation*}

To quantify the approximation error $\xi_n$, certain knowledge of approximation theory is required. 
There is rich literature on the approximation properties of neural networks. For instance, \cite{cheang2000better} and \cite{cheang2010approximation} provided tight approximation error bound for simple indicator functions; \cite{ismailov2017approximation} studied approximation efficiency of shallow neural network. Some recent works characterize the approximation accuracy of sparsely connected deep nets \cite{Bolcskei2019, Schmidt-Hieber2017Nonparametric, Bauler2019On} as well.

The following lemma is due to \cite[Theorem  3]{Schmidt-Hieber2017Nonparametric}.
\begin{lem}\label{lmsch1}
Assume $f_0 \in \mathcal{H}^{\alpha}_p$ for some $\alpha > 0$, then there exists a neural net $\widehat{f} \in \mathcal{F}(L,\boldsymbol{p}, s)$ with $\boldsymbol{p}=(12pN, \ldots, 12pN)\in\mathbb{R}^{L}$ whose bias and weight parameters are bounded by 1, and
\begin{equation}\label{choices}
\begin{split}
& L = 8 + (\lfloor \log_2n \rfloor+5)(1+\lceil \log_2 p \rceil),\\
& s \leq 94p^2(\alpha+1)^{2p}N(L+\lceil \log_2 p\rceil), \\
& N =C_N\lfloor n^{p/(2\alpha+p)}/\log(n)\rfloor,
\end{split}
\end{equation}
for some positive constant $C_N$, such that 
\begin{equation}{\label{eqsch1}}
\begin{split}
    \|\widehat{f}-f_0\|_{\infty}
    \leq (2\|f_0\|^{\alpha}_{\mathcal{H}}+1)3^{p+1}\frac{N}{n}+\|f_0\|^{\alpha}_{\mathcal{H}}2^{\alpha}(N)^{-\alpha/p}.
\end{split}
\end{equation}
\end{lem}
Lemma $\ref{lmsch1}$ summarizes the expressibility of sparse ReLU DNN in terms of its depth, width and sparsity.
It trivially implies that if $L,N,s$ satisfy (\ref{choices}) and $p=O(1)$, then $\max(\xi_n,r_n,\epsilon_n^2)= O(n^{2\alpha/(2\alpha+p)}\log^{\delta} n)$ for some $\delta>1$. Therefore, Theorem \ref{thm:d2bound} implies the following corollary.
\begin{cor}\label{corminimax}
Assume $f_0 \in \mathcal{H}^{\alpha}_p$ for some known $\alpha > 0$, where $p = O(1)$. Choose $L$, $s$ and $N$ as in (\ref{choices}). Then, our variational modeling satisfies that
\begin{equation}\label{postcon}
\begin{split}
&\int d^2(P_{\theta}, P_0)\widehat{q}(\theta)d\theta\leq C''[n^{-\alpha/(2\alpha+p)}\log^{\delta}(n)]^2,
\end{split}
\end{equation}
with dominating probability, for some $\delta>1$ and some constant $C''>0$.
\end{cor}

Corollary \ref{corminimax} establishes the rate minimaxity (up to a logarithmic factor) of variational sparse DNN inference. The established rate matches the contraction rate of the true Bayesian posterior (\cite{Polson2018posterior}) and therefore implies that there is no sacrifice in statistical rate with variational inference. Note that
(\ref{postcon}) also implies that the VB posterior mass of $\{d(P_\theta,P_0)\geq C''n^{-\alpha/(2\alpha+p)}
\log^{\delta}(n)\}$ converges to zero in probability, hence almost all of the VB posterior mass contracts towards a small Hellinger ball with (near-) minimax radius centered at $P_0$. 

The choices of $N$ and $s$ in (\ref{choices}), although lead to rate-minimaxity, relies on the smoothness parameter $\alpha$ which is usually unknown in practice. Therefore, the adaptive variational modeling discussed in Section \ref{sec:adapt} can be implemented here to select a reasonable $N$ and $s$ adaptively, such that the rate (near-)minimax convergence still holds.

\begin{cor}\label{cor2}
Assume $f_0 \in \mathcal{H}^{\alpha}_p$ for some unknown $\alpha > 0$, where $p = O(1)$. Choose $L$ as in (\ref{choices}) and let $N$ and $s$ follow the prior (\ref{prior2}). Then result (\ref{postcon}) still holds for the adaptive variational approach.
\end{cor}

\subsection{Teacher-student framework} \label{sec:teacher}
Under the H{\"o}lder smooth assumption, the rate of convergence $n^{-\alpha/(2\alpha+p)}$ suffers from the curse of dimensionality. 
Note that this rate merely represents the worse-case analysis among all H{\"o}lder smooth functions, which may not be suitable for real structured dataset. Hence, in this section, we are interested in the teacher-student framework, i.e., the underlying $f_0$ is exactly an unknown fixed sparse ReLU network (so-called teacher network), that is, $f_0\in\mathcal{F}(L_0,\boldsymbol{p}_0,s_0)$ for some $L_0$, $\boldsymbol{p}_0=(p_{0,1},\dots,p_{0,L_0})'$ and $s_0$, and its network parameter is denoted by $\theta_0$.

Our variational Bayes modeling with spike and slab prior can be used to train the so-called student network, based on data generated by the  teacher network. Adopting this teacher-student framework can better facilitate the understanding of how deep neural networks work in high-dimensional data as it provides an explicit target function with bounded complexity.

When certain information of teacher network structure is available, we have the following result.
\begin{cor}\label{cor3}
Under the teacher-student framework, if we choose  $L=L_0$, $s\geq s_0$ and $N\geq \max_{1\leq i\leq L_0} p_{0,i}/(12p)$, $B_0\geq\|\theta_0\|_\infty$ (under uniform prior) and
$\max\{Ls\log(pN), s\log(nL/s)\}=o(n)$ holds, then our variational Bayes approach satisfies
\begin{align}
\int &d^2(P_{\theta}, P_0)\widehat{q}(\theta)d\theta\leq C''\nonumber\\
&\qquad \qquad \quad \left({\frac{s\log(nL/s) + Ls\log(pN)}{n}}\log^{2\delta}(n)\right), \label{postcon2}
\end{align}
with dominating probability, for some constant $C''>0$ and any $\delta>1$.
\end{cor}
The choice of ($N, s$) means that we delibrately choose a wider and denser network structure, which ensures that the approximation error $\xi_n=0$. 

When the information of $s_0$ and $\boldsymbol{p}_0$ is not available, by adopting the adaptive variational modeling 
we also have the following result: 
\begin{cor}\label{cor4}
If the teacher network structure satisfies that
$\max\{L_0s_0\log(p\max p_{0,i}), s_0\log(nL_0/s_0)\}=o(n^{\alpha})$ for some $\alpha\in(0,1)$, and we choose $L=L_0$, 
and let $N$ and $s$ follow the prior (\ref{prior2}), $B_0\geq\|\theta_0\|_\infty$ (under uniform prior), then our adaptive variational Bayes approach satisfies
\begin{align}
&\int d^2(P_{\theta}, P_0)\widehat{q}(\theta)d\theta \leq C''\nonumber\\
&\: \left({\frac{s_0\log(nL_0/s_0) + L_0s_0\log(p\max p_{0,i})}{n}}\log^{2\delta}(n)\right), \label{postcon3}
\end{align}
with dominating probability, for any $\delta>1$ and some constant $C''>0$.
\end{cor}

The above two corollaries show that, under the teacher-student framework, the input dimension $p$ (i.e., input layer width) and hidden layer width $\boldsymbol{p_0}$ have at most logarithmic effect on the VB posterior convergence rate. Therefore, it doesn't suffer from the curse of dimensionality.

\section{Convergence under $L_2$ Norm} \label{sec:gen}
Our main theorems \ref{thm:d2bound} and \ref{adapt} concern the posterior convergence with respect to the Hellinger metric. Although commonly used in the Bayesian literature (\cite{ghosal2007convergence,Pati2018on, Zhang2019Convergence}),  Hellinger distance is of less practical interest than $L_2$ norm, i.e., $\mathbb{E}_X|f_{\theta}(X) - f_0(X)|^2$, for regression problems. However, a result directly addressing the $L_2$ convergence may not be reasonable due to the extreme flexibility of DNN models. For instance, given $p=1$, two ReLU DNN networks $f_\theta(x) \equiv 0$ and $f_{\theta'}(x) \equiv M\sigma (x-1+\varepsilon)$ can have arbitrarily large $L_2$ distance when $M$ is sufficiently huge, but are impossible to be discriminated when $\varepsilon$ is so tiny that no sampled $X_i$ visits the interval $[1-\varepsilon,1]$. 

Accordingly, our $L_2$ convergence result will exclude the ``irregular'' DNN model $f_\theta$'s whose $L_2$ distances from $f_0$ are mostly contributed by the integral of $[f_\theta(x)-f_0(x)]^2$ over some tiny-measure subset of $[-1,1]^p$.
To be more precise, we define the $L_2$ distance between $f_{\theta}$ and $f_0$ as $L_2^2(f_\theta,f_0)=\mathbb{E}_X|f_{\theta}(X)-f_{0}(X)|^2$, and let $\mathcal{G}\subset\mathcal F(L, \boldsymbol{p}, s )$ be the subset class of all ``regular'' DNNs that satisfy
\[
\mathbb{E}_{X}\{|f_{\theta}(X)-f_0(X)|^2 1(X \in \mathcal S)\}
\geq \kappa L^2_2(f_0,f_\theta),
\]
for some constant $0<\kappa\leq 1$, where  
\[
\mathcal S =\{X:|f_{\theta}(X)-f_0(X)|^2\leq \gamma_nL^2_2(f_0,f_\theta) \},
\]
for some $\gamma_n\rightarrow \infty$.
 $\mathcal{G}$ represents the DNNs that possesses a large enough expected square $L_2$ distance between $f_{\theta}$ and $f_0$ on a set $\mathcal S$ where $|f_{\theta}(X) - f_0(X)|^2$ is upper bounded, and the integral of $[f_\theta(x)-f_0(x)]^2$ over $\mathcal S^c$ doesn't make dominating contribution to $L^2_2(f_0,f_\theta)$. Naturally, $\mathcal G$ excludes the cases when $L^2_2(f_{\theta}, f_0)$ is mainly determined by the data from only a small set of the support of $X$.

Let $\widetilde \varepsilon_n^2$ denote the Hellinger convergence rate in Theorem \ref{thm:d2bound} or \ref{adapt}, i.e., $\widetilde\varepsilon^2_n$ is of the same order as the RHS of equation (\ref{d2bound}) or (\ref{d2bound+}).
We have the following convergence result regarding $L_2$ metric, which states that the variational posterior mass over the irregular DNNs, which have $L_2$ error greater than $M_n \widetilde \varepsilon^2_n$, is negligible. 

\begin{thm} \label{thm:gen}
Given any pre-specified network family as Theorem \ref{thm:d2bound} or under the adaptive variational Bayes modeling as Theorem \ref{adapt}, if $\gamma_n\widetilde \varepsilon^2_n=o(1)$, then we have that w.h.p.
\[
\int_{\mathcal{G}\cap\{L^2_2(f_0, f_{\theta})\geq M_n\widetilde \varepsilon_n^2\} } \widehat q(\theta)d\theta =o(1),
\]
for any sequence $M_n\rightarrow\infty$.
\end{thm}

\begin{remark}
In the literature, there do exist some direct results regarding $L_2$ convergence rate of DNN learning and these results usually rely on some regularity condition such as the $L_\infty$ boundedness of DNNs in the model space (\cite{Schmidt-Hieber2017Nonparametric, Polson2018posterior}).  However, in practice, it is usually infeasible to ensure that the trained DNN models meet the pre-specified bound, since the relationship between the magnitude of $\theta$ and $|f_\theta|_\infty$ is rather complicated.
\end{remark}



\section{Experiments}\label{exp}
\begin{table*}[!ht]
\centering
\caption{Results for teacher network experiment. The average test RMSE with standard error and average posterior number of edges with standard error are exhibited.}
\label{tb:teacher}
\begin{tabular}{lcccccccc}
\toprule
 & \multicolumn{4}{c}{\textbf{Test RMSE}} & \multicolumn{4}{c}{\textbf{\# of edges}} \\ \cmidrule(lr){2-5} \cmidrule(lr){6-9} 
\textbf{Width} & \textbf{ASVI}& \textbf{SVI} & \textbf{HS-BNN} & \textbf{Dense-BNN} & \textbf{ASVI} & \textbf{SVI}& \textbf{HS-BNN} & \textbf{Dense-BNN} \\ 
\midrule
2 &-& 2.193 $\pm$ 0.195 & 2.193 $\pm$ 0.163 & 2.131 $\pm$ 0.097 &-& 48.28 $\pm$ 2.099 & 51.00 $\pm$ 0.000 &51.00 $\pm$ 0.000  \\
4 &-& 1.636 $\pm$ 0.069 & 1.715 $\pm$ 0.160 & 1.591 $\pm$ 0.087 &-& 94.43 $\pm$ 4.499 & 109.0 $\pm$ 0.000 &109.0 $\pm$ 0.000  \\
6 &-& 1.210 $\pm$ 0.049 & 1.322 $\pm$ 0.179 & 1.190 $\pm$ 0.033 &-& 125.7 $\pm$ 8.805 & 175.0 $\pm$ 0.000 &175.0 $\pm$ 0.000  \\
8 &-& 1.065 $\pm$ 0.038 & 1.108 $\pm$ 0.048 & 1.046 $\pm$ 0.021 &-& 135.5 $\pm$ 10.87 & 249.0 $\pm$ 0.000 &249.0 $\pm$ 0.000  \\
10 &-& 1.014 $\pm$ 0.023 & 1.058 $\pm$ 0.029 & 1.014 $\pm$ 0.010 &-& 151.1 $\pm$ 13.25 & 331.0 $\pm$ 0.000 &331.0 $\pm$ 0.000  \\
12 &-& 1.019 $\pm$ 0.085 & 1.035 $\pm$ 0.016 & 1.010 $\pm$ 0.007 &-& 166.1 $\pm$ 14.41 & 421.0 $\pm$ 0.000 &421.0 $\pm$ 0.000  \\
14 &-& 1.018 $\pm$ 0.093 & 1.034 $\pm$ 0.010 & 1.011 $\pm$ 0.009 &-& 177.3 $\pm$ 15.62 & 519.0 $\pm$ 0.000 &519.0 $\pm$ 0.000  \\
16 &-& 1.011 $\pm$ 0.037 & 1.032 $\pm$ 0.010 & 1.009 $\pm$ 0.005 &-& 186.1 $\pm$ 16.48 & 625.0 $\pm$ 0.000 &625.0 $\pm$ 0.000  \\
18 &-& 1.005 $\pm$ 0.008 & 1.030 $\pm$ 0.010 & 1.010 $\pm$ 0.005 &-& 190.3 $\pm$ 15.87 & 739.0 $\pm$ 0.000 &739.0 $\pm$ 0.000  \\
20 &-& 1.003 $\pm$ 0.006 & 1.029 $\pm$  0.008 & 1.010 $\pm$ 0.007 &-& 192.5 $\pm$ 13.78 & 861.0 $\pm$ 0.000 &861.0 $\pm$ 0.000  \\ \midrule
Adaptive & 1.003 $\pm$ 0.010 & - & - & -& 155.9 $\pm$ 15.58 & - & - & -\\ 
\bottomrule
\end{tabular}
\end{table*}
In this section, we investigate the performance of the proposed Adaptive Sparse Variational Inference (ASVI) with Gaussian slab prior through empirical studies. To implement ASVI, after pre-specifying the depth $L$, one needs to assign prior distributions for $N$ and $s$ according to (\ref{prior2}), and assign uniform prior (\ref{prior}) over the network structure $\gamma$ given $s$. However, as emphasized in the introduction, it is not computationally feasible to solve ASVI, since the exact minimization of negative ELBO requires exhaustively search over all possible sparse network structures.
As a consequence, in this numerical studies section, an approximated solution of $\widehat q$ is used instead. The details of the approximation and implementation of ASVI are presented in Section 2 of the supplementary document. In short words, we integrate out the sparsity variable $s$ in the hierarchical prior (\ref{prior2}), and only consider the marginal modelling of $N$ and $\theta$. Given the width multiplier $N$, the maximized ELBO $\Omega(N)$ is obtained by back propagation with the help of some approximation and binary relaxation. The optimal structure is then selected by the penalized ELBO criterion similar to (\ref{eq:pelbo}). In this simulation, we typically specify 5-10 levels of width choices and compute $\Omega(N)$ for different $N$ in parallel.


For all the numerical studies, we use the VB posterior mean estimator $\widehat{f}=\sum^{30}_{i=1}f_{\theta_i}/30$ to assess the prediction accuracy, where $\theta_i$'s are randomly drawn from the VB posterior $\widehat{q}(\theta)$. We use $\widehat{s}=\sum^H_{i=1}\nu_i/H$ to measure the posterior network sparsity. We compare our method to Horseshoe BNN (HS-BNN) \cite{ghosh2017Model} and dense BNN \cite{Blundell2015weight}. 


\begin{figure}[h!]
    \centering
    \begin{subfigure}[b]{.24\textwidth}
        \centering
        \includegraphics[width=1\linewidth]{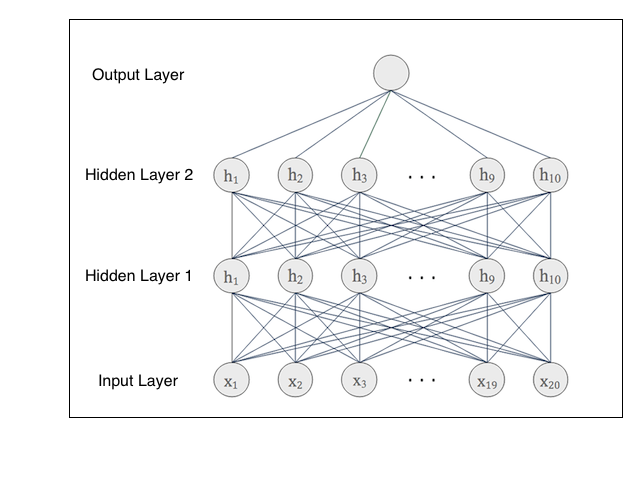}
        \caption{Teacher network}
    \end{subfigure}%
    \begin{subfigure}[b]{.24\textwidth}
        \centering
        \includegraphics[width=1\linewidth]{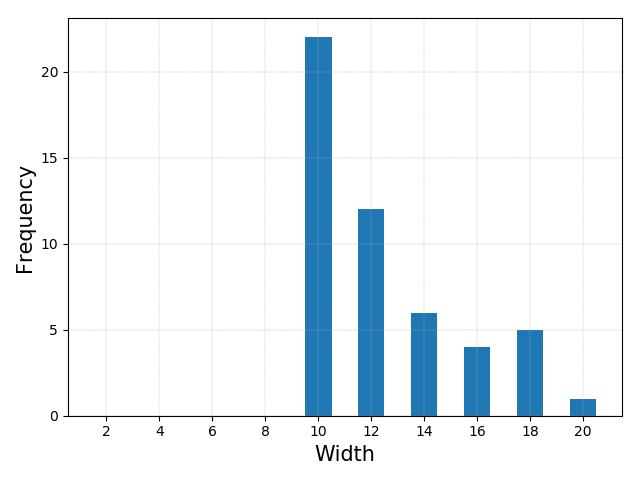}
        \caption{Selected width}
    \end{subfigure}
    \caption{(a) Teacher network with structure 20-10-10-1, where 50\% of the edges are set to 0 randomly. (b) Frequency of the selected width in 50 replications.}
    \label{fig:frequency}
\end{figure}

\begin{table}[h!]
\centering
\caption{Average test RMSE with standard error for UCI regression datasets.}
\label{tb:UCI}
\begin{tabular}{lccccc}
\toprule

\textbf{Dataset} & \textbf{n (p)} & \textbf{SVI} &\textbf{HS-BNN} & \textbf{PBP}   \\ 
\midrule
Kin8nm & 8192 (8) &0.08$\pm$0.00 & 0.08$\pm$0.00  & 0.10$\pm$0.00 \\
Naval & 11934 (16) & 0.00$\pm$0.00 & 0.00$\pm$0.00  & 0.01$\pm$0.00  \\
Power Plant & 9568 (4) & 4.02$\pm$0.18 & 4.03$\pm$0.15 &4.12$\pm$0.03  \\
Protein & 45730 (9) &4.36$\pm$0.04  & 4.39$\pm$0.04  & 4.73$\pm$0.01   \\
Wine & 1599 (11) & 0.62$\pm$0.03 &0.63$\pm$0.04  & 0.64$\pm$0.01 \\
Year & 515345 (90) & 8.85$\pm$NA & 9.26$\pm$NA  & 8.88$\pm$NA  \\ 
\bottomrule
\end{tabular}
\end{table}

\subsection{Simulation study}
We consider a simulated experiment under the teacher-student framework. As shown in Fig \ref{fig:frequency} (a), we use a 2-hidden-layer teacher network with ReLU activation, where the specific structure is 20-10-10-1. The edges of the teacher network are first randomly generated from $\mathcal{U}(0.5, 1.5)$ and then randomly set to 0 by a rate of 50\% to ensure a sparse structure. 
We fix the depth $L$ of student net to 2 in the experiment, and consider the width of student net to range from 2 to 20 with a increment of 2. We randomly generate 50 datasets of size $10000$ from the teacher network with random noise variance $\sigma_{\epsilon}=1$ for training, and the adaptive variational inference is performed on each of these datasets to select the best network structure. The remaining implementation details can be found in the supplementary document.

Fig \ref{fig:frequency} (b) plots the frequency of the selected width among the 50 replications. It shows that in most time the ASVI selects width 10 or 12, which is close to the true width. We compare the test Root Mean Squared Error (RMSE) of ASVI against non-adaptive SVI (i.e., ASVI without width selection), HS-BNN and Dense-BNN with all the choices of width. The result is displayed in Table \ref{tb:teacher}. It shows that ASVI achieves best test Root Mean Squared Error (RMSE), which is quite close to the random noise ($\sigma_{\epsilon}=1$). In addition, the number of edges selected by ASVI is also close to the ground truth (around 165.5).


\subsection{Real data}
We compare the performance of our method to others on UCI regression tasks and MNIST data. For UCI datasets, following the same experimental protocol as \cite{Lobato2015Probabilistic}, a single layer neural network of 50 units with ReLU activation is used for all the datasets, except for the larger ones "Protein" and "Year", where 100 units are used. For the smaller datasets, we randomly select 90\% and 10\% for training and testing respectively, and the process is repeated for 20 times. For "Protein", only 5 replication is performed. For "Year", where the training and testing datasets are predefined, the process is only done once. We compare our method to HS-BNN and probabilistic backpropagation (PBP) of \cite{Lobato2015Probabilistic}. For MNIST, we use a two hidden layer ReLU network with width of $\{400, 500, 600, 700, 800\}$. Other Implementation details can be found in the supplementary document.

Table \ref{tb:UCI} shows our method (SVI) performs as well as or better than the other methods on UCI datasets with pre-determined architecture. Figure \ref{fig:mnist} shows our method achieves best test accuracy for MNIST data, with a selected width of 700 and posterior sparsity of $6.01 \%$ (62855 edges) at epoch 300.

\begin{figure}[h!]
    \centering
    \includegraphics[width=0.9\linewidth]{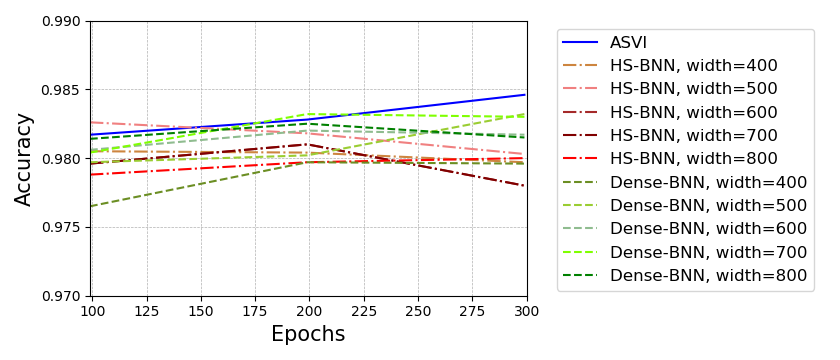}
    \caption{Test accuracy for MNIST data}
    \label{fig:mnist}
\end{figure}

\section{Summary}\label{dis}
In this work, we investigate the theoretical aspects of variational inference for sparse DNN models. 
Although theoretically sound, the spike and slab modeling with Dirac spike is difficult to implement in practice, and some continuous relaxation is required that deserves further theoretical investigation. In addition, despite the fact that the proposed uniform prior distribution for $s$ guarantees good theoretical properties, it is also not practical and some approximation is involved in our implementation. Therefore, some alternative choice of prior distribution could be investigated in the future.


%



\ifCLASSOPTIONcompsoc
  \section*{Acknowledgments}
\else
  \section*{Acknowledgment}
\fi

This work was completed in the fall of 2019 when Cheng was a member of Institute for Advanced Study, Princeton. Cheng acknowledges the hospitality of IAS and also financial support from Adobe Data Science Grant.
Dr. Song's research is partially supported by National Science Foundation grant DMS-1811812.

\ifCLASSOPTIONcaptionsoff
  \newpage
\fi



\bibliography{ref}

\appendices
\section{Technical Details}

The detailed proofs for our lemmas and theorems are included in this section. 

Throughout this section, make the following notations: denote the $n$ independent observations of $Y$ as $\boldsymbol{Y}^{(n)}=(Y_1, \dots, Y_n)$, denote $f_0(\boldsymbol{X}^{(n)})=(f_0(X_1), \ldots, f_0(X_n))$ and $f_{\theta}(\boldsymbol{X}^{(n)})=(f_{\theta}(X_1), \ldots, f_{\theta}(X_n))$.

\subsection{Proof of Lemma 4.1}

Lemma \ref{lmdonsker} restates the  Donsker and Varadhan's representation for the $\mbox{KL}$ divergence, its proof can be found in \cite{Boucheron2013Concentration}. 

\begin{lem}{\label{lmdonsker}}
For any probability measure $\lambda$ and any measurable function $h$ with $e^h \in L_1(\lambda)$,
$$
\log \int e^{h(\eta)}\lambda(d\eta) = \sup_{\rho}\left[\int h(\eta)\rho(d \eta) - \mbox{KL}(\rho\|\lambda).\right]
$$
\end{lem}

The next lemma proves the existence of a testing function which can exponentially separate $P_0$ and $\{P_\theta: d(P_0,P_\theta)\geq \varepsilon_n, P_\theta\in \mathcal{F}(L,\boldsymbol{p},s)\}$. The existence of such testing function is crucial for Lemma 4.2.

\begin{lem}{\label{lmtesting}}
Let $\varepsilon_n=M\sqrt{\frac{s\log(nL/s) + Ls\log(pN)}{n}}\log^\delta(n)$ for any $\delta\geq1$ and some large constant M. Then there exists some testing function $\phi\in[0,1]$ and $C_1>0$, $C_2>1/3$, such that
\[
\begin{split}
\mathbb{E}_{P_0}(\phi)&\leq  \exp\{-C_1n\varepsilon_n^2\},\\
\sup_{\substack{P_{\theta} \in \mathcal{F}(L,\boldsymbol{p},s)\\ d(P_{\theta}, P_0)>\varepsilon_n}}\mathbb{E}_{P_{\theta}}(1-\phi)&\leq\exp \{-C_2nd^2(P_0,P_{\theta})\}.
\end{split}
\]
\end{lem}
\begin{proof}
Due to the well-known result (e.g., \cite{le2012asymptotic}, page 491 or \cite{ghosal2007convergence}, Lemma 2), there always exists a function $\psi\in[0,1]$, such that
\[
\begin{split}
&\mathbb{E}_{P_0}(\psi)\leq \exp\{-nd^2(P_{\theta_1},P_{0}) /2\},\\
&\mathbb{E}_{P_{\theta}}(1-\psi)\leq \exp\{-nd^2(P_{\theta_1},P_{0}) /2\},
\end{split}
\]
for all $P_\theta\in\mathcal{F}(L,\boldsymbol{p},s)$ satisfying that $d(P_{\theta},P_{\theta_1})\leq d(P_{0},P_{\theta_1})/18$. 

Let $K=N(\varepsilon_n/19,\mathcal{F}(L,\boldsymbol{p},s),d(\cdot,\cdot) )$ denote the covering number of set $\mathcal{F}(L,\boldsymbol{p},s)$, i.e., there exists $K$ Hellinger-balls with radius $\varepsilon_n/19$, that completely cover $\mathcal{F}(L,\boldsymbol{p},s)$.
For any $\theta\in\mathcal{F}(L,\boldsymbol{p},s)$ (W.O.L.G, we assume $P_\theta$ belongs to the $k$th Hellinger ball centered at $P_{\theta_k}$), if $d(P_\theta, P_0)>\varepsilon_n$,
then  we must have that  $d(P_0,P_{\theta_k})>(18/19)\varepsilon_n$ and there exists a testing function $\psi_k$, such that
\[
\begin{split}
\mathbb{E}_{P_0}(\psi_k)&\leq \exp\{-nd^2(P_{\theta_k},P_{0}) /2\}\\
&\leq \exp\{-(18^2/19^2/2)n\varepsilon_n^2\},\\
\mathbb{E}_{P_{\theta}}(1-\psi_k)&\leq \exp\{-nd^2(P_{\theta_k},P_{0}) /2\}\\
&\leq \exp\{-n(d(P_0,P_{\theta})-\varepsilon_n/19)^2  /2\}\\
&\leq \exp\{-(18^2/19^2/2)nd^2(P_0,P_{\theta})\}.
\end{split}
\]
Now we define $\phi=\max_{k=1,\dots,K}\psi_k$.
Thus we must have
\[
\begin{split}
\mathbb{E}_{P_0}(\phi)&\leq\sum_k \mathbb{E}_{P_0}(\psi_k)\leq K \exp\{-(18^2/19^2/2)n\varepsilon_n^2\}\\
&\leq \exp\{-((18^2/19^2/2)n\varepsilon_n^2-\log K)\}.
\end{split}
\]
Note that
\begin{align}
&\log K=\log N(\varepsilon_n/19,\mathcal{F}(L,\boldsymbol{p},s),d(\cdot,\cdot) )\nonumber\\
&\leq \log N(\sqrt{8}\sigma_{\varepsilon}\varepsilon_n/19,\mathcal{F}(L,\boldsymbol{p},s),\|\cdot\|_\infty )\nonumber\\
&\leq (s+1)\log(\frac{38}{\sqrt{8}\sigma_{\varepsilon}\varepsilon_n}L(12pN+1)^{2L})\nonumber\\
&\leq s\log \frac{1}{\varepsilon_n} + s\log(nL/s) + sL\log(pN)\nonumber\\
&\leq n\varepsilon_n^2/4,\label{testrate}
\quad \mbox{ for sufficiently large n} ,
\end{align}
where the first inequality is due to the fact
\begin{equation*}\begin{split}
d^2(P_{\theta}, P_0) 
\leq 1- \exp\{-\frac{1}{8\sigma^2_{\epsilon}}\|f_0 - f_{\theta}\|^2_\infty\}
\end{split}
\end{equation*}
and $\varepsilon_n=o(1)$,
the second inequality is due to Lemma 10 of \cite{Schmidt-Hieber2017Nonparametric}.
Therefore,
\[
\begin{split}
\mathbb{E}_{P_0}(\phi)&\leq\sum_k P_0(\psi_k)\leq  \exp\{-C_1n\varepsilon_n^2\},
\end{split}
\]
for some $C_1=18^2/19^2/2-1/4$. On the other hand, for any $\theta$, such that $d(P_\theta,P_0)\geq \varepsilon_n$, say $P_\theta$ belongs to the $k$th Hellinger ball, then we have
\[
\begin{split}
\mathbb{E}_{P_{\theta}}(1-\phi)&\leq \mathbb{E}_{P_{\theta}}(1-\psi_k)\leq\exp \{-C_2nd^2(P_0,P_{\theta})\},
\end{split}
\]
where $C_2=18^2/19^2/2$. Hence we conclude the proof.
\end{proof}

{\noindent\bf Proof of Lemma 4.1}
\begin{proof}
It suffices to construct some $q^*(\theta)\in \mathcal Q$, such that w.h.p,
\begin{align}
&\mbox{KL}(q^*(\theta)\|\pi(\theta)) + \int l_n(P_0, P_{\theta})q^*(\theta)d\theta \nonumber\\
\leq &nr_n+\frac{3n}{2\sigma_{\varepsilon}^2}\inf_{\theta}\|f_{\theta}-f_0\|^2_{\infty}+\frac{3nr_n}{2\sigma^2_{\epsilon}}.\label{eq:lemma4.1}
\end{align}
Let $\theta^{\ast}=\arg\min_{\theta \in \Theta_B(L,\boldsymbol{p},s)} \|f_{\theta}-f_0\|^2_2$ and we choose the same $q^*(\theta)$ that has been used in the proof of Theorem 2 of \cite{Cherief2019Convergence}. Specifically,  for all $h=1,\ldots, H$, $\gamma^{\ast}_{h}=\mathbb{I}(\theta^{\ast}_h \neq 0)$, and\\
i) For uniform slab distribution,
\begin{equation}\label{qstar}
\begin{split}
&\theta_h \sim \gamma^{\ast}_h\mathcal{U}([\theta^{\ast}_h-a_n, \theta^{\ast}_h+a_n]) + (1-\gamma^{\ast}_h)\delta_{0},
\end{split}
\end{equation}
where $a_n=\frac{s}{4n}(12BpN)^{-2L}\{(p+1+\frac{1}{12BpN-1})^2 \frac{L^2}{(12BpN)^2}+\frac{1}{(12BpN)^2-1}+\frac{2}{(12BpN-1)^2}\}^{-1}$. \\
ii) For Gaussian slab distribution,
\begin{equation}\label{qstar2}
\begin{split}
&\theta_h \sim \gamma^{\ast}_h\mathcal{N}(\theta^*_h, \sigma^2_n) + (1-\gamma^{\ast}_h)\delta_{0},
\end{split}
\end{equation}
where $\sigma^2_n=\frac{s}{16n}\log(36pN)^{-1}(24BpN)^{-2L}\\
\{(p+1+\frac{1}{12BpN-1})^2+\frac{1}{(24BpN)^2-1}+\frac{2}{(24BpN-1)^2}\}^{-1}$.

According to the proof of Theorem 2 in \cite{Cherief2019Convergence},

\begin{align}
&\mbox{KL}(q^*(\theta)\|\pi(\theta)) \leq nr_n, \label{eqche2}\\
&\int \|f_{\theta} - f_{\theta^{\ast}}\|_{\infty}^2q^*(\theta)d\theta \leq r_n, \label{eqche1}
\end{align}
and the first term on L.H.S of (\ref{eq:lemma4.1}) is bounded.

To upper bound the second term on L.H.S of (\ref{eq:lemma4.1}), note that
\[
\begin{split}
l_n(P_0, P_{\theta}) = &\frac{1}{2\sigma^2_{\epsilon}}(\|\boldsymbol{Y}^{(n)} - f_{\theta}(\boldsymbol{X}^{(n)})\|^2_2 \\
&- \|\boldsymbol{Y}^{(n)} - f_{0}(\boldsymbol{X}^{(n)})\|^2_2)\\
= &\frac{1}{2\sigma^2_{\epsilon}} (\|\boldsymbol{Y}^{(n)} - f_{0}(\boldsymbol{X}^{(n)})+ f_0(\boldsymbol{X}^{(n)})\\
&-f_{\theta}(\boldsymbol{X}^{(n)}))\|^2_2 - \|\boldsymbol{Y}^{(n)} - f_{0}(\boldsymbol{X}^{(n)})\|^2_2)\\
= &\frac{1}{2\sigma^2_{\epsilon}}(\|f_{\theta}(\boldsymbol{X}^{(n)})-f_0(\boldsymbol{X}^{(n)})\|^2_2 \\
&+ 2\langle \boldsymbol{Y}^{(n)}-f_0(\boldsymbol{X}^{(n)}), f_0(\boldsymbol{X}^{(n)}) - f_\theta(\boldsymbol{X}^{(n)})\rangle).
\end{split}
\]
Denote
\[
\begin{split}
\mathcal{R}_1 &= \int \|f_{\theta}(\boldsymbol{X}^{(n)}) - f_0(\boldsymbol{X}^{(n)})\|^2_2 q^*(\theta)d\theta, \\
\mathcal{R}_2 &=\int \langle \boldsymbol{Y^{(n)}}-f_0(\boldsymbol{X}^{(n)}), f_0(\boldsymbol{X}^{(n)}) - f_\theta(\boldsymbol{X}^{(n)})\rangle q^*(\theta)d\theta.
\end{split}
\]
Since $\|f_{\theta}(\boldsymbol{X}^{(n)}) - f_0(\boldsymbol{X}^{(n)})\|^2_2 \leq n \|f_{\theta} - f_0\|^2_{\infty}\leq n\|f_{\theta}-f_{\theta^{\ast}}\|^2_{
\infty}+n\|f_{\theta^{\ast}}-f_0\|^2_{\infty}$,
\[
\mathcal{R}_1 \leq  nr_n + n\|f_{\theta^{\ast}}-f_0\|^2_{\infty} .
\]
Noting that $\boldsymbol{Y}^{(n)} - f_0(\boldsymbol{X}^{(n)}) = \epsilon \sim \mathcal{N}(0, \sigma^2_{\epsilon}I)$, then
\[
\begin{split}
\mathcal{R}_2 &= \int \epsilon^T(f_0(\boldsymbol{X}^{(n)}) - f_\theta(\boldsymbol{X}^{(n)}))q^*(\theta)d\theta\\
&= \epsilon^T\int (f_0(\boldsymbol{X}^{(n)}) - f_\theta(\boldsymbol{X}^{(n)}))q^*(\theta)d\theta\\
&\sim \mathcal{N}(0, c_f\sigma^2_{\epsilon}),
\end{split}
\]
where $c_f = \|\int (f_0(\boldsymbol{X}^{(n)}) - f_\theta(\boldsymbol{X}^{(n)}))q^*(\theta)d\theta\|^2_2 \leq \mathcal{R}_1$ due to Cauchy-Schwarz inequality. Then by Gaussian tail bound
\[
P_0(\mathcal{R}_2 \geq \mathcal{R}_1) \leq \exp(-\frac{\mathcal{R}^2_1}{2\sigma^2_{\epsilon}\mathcal{R}_1}),
\]
which implies $\mathcal{R}_2 \leq \mathcal{R}_1$ w.h.p..
Therefore,
\[
\begin{split}
\int l_n(P_0, P_{\theta})q^*(\theta)d\theta&= \mathcal{R}_1/2\sigma^2_{\epsilon} + \mathcal{R}_2/\sigma^2_{\epsilon}\\
&\leq 3n (r_n + \|f_{\theta^{\ast}}-f_0\|^2_{\infty})/2\sigma_{\varepsilon}^2  \mbox{,  w.h.p.}, 
\end{split}
\]
which concludes this lemma together with (\ref{eqche2}).

\end{proof}

\subsection{Proof of Lemma 4.2}
The proof is adapted from the proof of Theorem 3.1 in \cite{Pati2018on}.
\begin{proof}
We claim that with high probability (w.h.p), 
\begin{equation}\label{eq1}
M = \int_{\Theta} \eta(P_{\theta}, P_0)\pi(\theta)d\theta \leq e^{Cn\varepsilon_n^2}
\end{equation} 
for some $C>0$, where $\log\eta(P_{\theta}, P_0)=l_n(P_{\theta}, P_0)+\frac{n}{3}d^2(P_{\theta}, P_0)$. Thus by Lemma $\ref{lmdonsker}$, w.h.p.,
\[
\begin{split}
&\frac{n}{3}\int d^2(P_{\theta}, P_0)\widehat q(\theta)d\theta \\
\leq  & Cn\varepsilon_n^2 + \mbox{KL}(\widehat q(\theta)\|\pi(\theta)) - \int l_n(P_{\theta}, P_0)\widehat q(\theta)d\theta\\
\leq  & Cn\varepsilon_n^2 + \mbox{KL}( q(\theta)\|\pi(\theta)) - \int l_n(P_{\theta}, P_0) q(\theta)d\theta
\end{split}
\]
holds for any distribution $q_\theta$. The last inequality holds since
that $\mbox{KL}(q(\theta)\|\pi(\theta)) - \int l_n(P_{\theta}, P_0)q(\theta)d\theta$ is the negative ELBO function up to a constant, which is minimized at $\widehat q(\theta)$. This concludes Lemma 4.3.

To prove ($\ref{eq1}$), we define
\begin{align*}
& M_1 = \int_{d(P_{\theta}, P_0)\leq \varepsilon_n}\eta(P_{\theta}, P_0)\pi(\theta)d\theta,\\
& M_2 = \int_{d(P_{\theta}, P_0)> \varepsilon_n}\eta(P_{\theta}, P_0)\pi(\theta)d\theta,
\end{align*}
and will bound both $M_1$ and $M_2$.

For $M_1$, by Fubini's theorem,
\begin{equation*}{\label{eqt1}}
\begin{split}
\mathbb{E}_{P_0}M_1=&\int_{d(P_{\theta}, P_0)\leq \varepsilon_n}\int \frac{p_{\theta}(\boldsymbol{D}^{(n)})}{p_0(\boldsymbol{D}^{(n)})}e^{\frac{n}{3}d^2(P_{\theta}, P_0)}\\
&dP_0(\boldsymbol{D}^{(n)})\pi(\theta)d\theta\\
=&\int_{d(P_{\theta}, P_0)\leq \varepsilon_n}e^{\frac{n}{3}d^2(P_{\theta}, P_0)}\pi(\theta)d\theta\\
\leq &e^{\frac{n}{3}\varepsilon_n^2}.
\end{split}
\end{equation*}
It follows from Markov inequality that $M_1 \leq e^{Cn\varepsilon_n^2}$ w.h.p..

For $M_2$, we further decompose it as $M_2 = M_{21} + M_{22}$, 
\begin{align*}
& M_{21}=\int_{d(P_{\theta}, P_0)> \varepsilon_n}\phi\eta(P_{\theta}, P_0)\pi(\theta)d\theta, \\
& M_{22}=\int_{d(P_{\theta}, P_0)> \varepsilon_n}(1-\phi)\eta(P_{\theta}, P_0)\pi(\theta)d\theta,
\end{align*}
where the testing function $\phi$ is defined in Lemma $\ref{lmtesting}$.

For $M_{21}$, since $\mathbb{E}_{P_0}[\phi]\leq e^{-C_1n\varepsilon_n^2}$,  $\phi \leq e^{-C'_1n\varepsilon_n^2}$ for some $C'_1>0$ w.h.p., thus $M_{21} \leq e^{-C'_1n\varepsilon_n^2}M_2$ w.h.p.

For $M_{22}$, by Fubini's theorem and Lemma $\ref{lmtesting}$,
\begin{align*}
\mathbb{E}_{P_{0}}M_{22}&=\int_{d(P_{\theta}, P_0)> \varepsilon_n}\mathbb{E}_{P_{\theta}}(1-\phi_n)e^{\frac{n}{3}d^2(P_{\theta}, P_0)}\pi(\theta)d\theta\\
&\leq e^{-(C_2-1/3)n\varepsilon_n^2}:=e^{-C_2'n\varepsilon_n^2}.
\end{align*}
Thus, $M_2 \leq e^{-C'_1n\varepsilon_n^2}M_2 + e^{-C'_2n\varepsilon_n^2}$ w.h.p., which implies that $M_2 \leq e^{-C^{''}_2n\varepsilon_n^2}$ w.h.p. for some $C^{''}_2>0$.

Combine the boundedness results for both $M_1$ and $M_2$, we conclude ($\ref{eq1}$).

\end{proof}

\subsection{Proof of Theorem 5.1}
The following Lemmas \ref{truncated} and \ref{prior3} consider the situation that the network width $N$ and $s$ are not specified. These two lemmas prepares our proof for Theorem 5.1.

\begin{lem}\label{truncated}
Let $N_n =c_N [Ls^*\log N^*+s^*\log(Ln/s^*)] \\ \log^{2\delta}(n)\asymp n\varepsilon_n^{*2}$ and $s_n\lambda_s= c_s [Ls^*\log N^*+s^*\log(Ln/s^*)] \\ \log^{2\delta}(n)\asymp n\varepsilon_n^{*2}$ for some constant $c_N$ and $c_s$ ($N^*$, $s^*$ and $\varepsilon^*_n$ are defined in Section 5).  If the neural network width $N$ and sparsity $s$ follow some truncated priors with support $\{1,\dots,N_n\}$ and $\{0,\dots,s_n\}$ respectively, and this prior satisfies $-\log \pi(N=N^*,s=s^*)=O(n\varepsilon_n^2)$. Then similar results of Lemma 4.1 and Lemma 4.2 holds, that is for some $C>0$ and $C'>0$, we have
 \begin{equation}\label{re}
\begin{split}
&\int d^2(P_{\theta}, P_0)\widehat{q}(\theta)d\theta \leq C\varepsilon_n^{*2} + \frac{3}{n}\inf_{q(\theta) \in \mathcal Q}\Bigl\{ \mbox{KL}(q(\theta)\|\pi(\theta))\\
&\qquad\qquad\qquad\qquad\qquad\quad + \int l_n(P_0, P_{\theta})q(\theta)d\theta \Bigr\}, \mbox{ and}\\
&\inf_{q(\theta) \in \mathcal Q}\Bigl\{ \mbox{KL}(q(\theta)\|\pi(\theta))
+ \int l_n(P_0, P_{\theta})q(\theta)d\theta \Bigr\} \\
&\qquad\qquad\qquad\qquad\qquad\qquad\qquad \leq C'n (\varepsilon_n^{*2}+r_n^*
+ \xi_n^* )
\end{split}
\end{equation}
hold with dominating probability.
\end{lem}


\begin{proof}
To prove the first result of (\ref{re}), similarly to the proof of Lemma 4.2, it is essential to show that there exists some testing function that achieves exponentially small error probability. This further requires a bounded covering number of $N(\varepsilon_n^*/19,\bigcup_{N=1}^{N_n}\bigcup_{s=0}^{s_n}\mathcal{F}(L,\boldsymbol{p}_N^{L},s),d(\cdot,\cdot))$.
Similar to (\ref{testrate}), we have that
\[
\begin{split}
&N(\varepsilon_n^*/19,\bigcup_{N=1}^{N_n}\bigcup_{s=0}^{s_n}\mathcal{F}(L,\boldsymbol{p}_N^{L},s),d(\cdot,\cdot))\\
\leq&\log N(\sqrt{8}\sigma_{\epsilon}\varepsilon_n^*/19,\bigcup_{N=1}^{N_n}\bigcup_{s=0}^{s_n}\mathcal{F}(L,\boldsymbol{p}_N^{L},s),\|\cdot\|_\infty )\\
\leq &\log (s_n)+\log (N_n) +\\ &(s_n+1)\log(\frac{38}{\sqrt{8}\sigma_{\epsilon}\varepsilon_n^*}L(12pN_n+1)^{2L})\\
\leq & n\varepsilon_n^{*2}/4,\quad\mbox{given a large n} ,
\end{split}
\]
where the last inequality holds due to the fact that $\log(N_n)\asymp \log n$, $s_n\log (1/\varepsilon_n^*)\asymp s_n\log n$ and 
$\lambda_s\geq aL\log n$ for some $a>0$.
Therefore, by the argument of Lemma \ref{lmtesting}, there still exists a testing function that separate $P_0$ and $\{P_\theta: d(P_0,P_\theta)\geq \varepsilon_n, P_\theta\in \bigcup_{N=1}^{N_n}\bigcup_{s=0}^{s_n}\mathcal{F}(L^*,\boldsymbol{p}_N^{L^*},s)\}$ with exponentially small error probability. By the argument used in the proof of Lemma 4.2, implies that first result of (\ref{re}) holds. 

The proof of the second result of (\ref{re}) follows the same argument used in Lemma 4.1. We can choose the $ q^*(\theta, N, s)\in\mathcal{Q}_{N,s}$ as $q^*(N)=\delta_{N^*}$, $q^*(s)=\delta_{s^*}$, and 
$q^*(\theta|N^*,s^*)=q^*(\theta)$ as defined in (\ref{qstar}). Trivially,
(\ref{eqche1}) still holds, and $\mbox{KL}(q^*(\theta, N, s)\|\pi(\theta, N, s))\leq nr_n^*-\log \pi(N=N^*,s=s^*)=O(n\varepsilon_n^{*2}+nr_n^*)$. It hence concludes the result.
\end{proof}

The next Lemma is an improved result of Corollary 6.1 in \cite{Polson2018posterior}.
\begin{lem}\label{prior3}
Under prior specification (13),
\begin{equation*}
\pi(N\geq N_n\mbox{ or }s\geq s_n|\boldsymbol{D}^{(n)})\leq \exp\{-c_0n\varepsilon_n^{*2}\}, 
\end{equation*}
where constant $c_0$ increases to infinity as $c_s$ (defined in Lemma \ref{truncated}) increases.
\end{lem}



\begin{proof}
Due to Lemma A.4 in \cite{song2017nearly}, it suffice to show that 
\begin{align}
&\pi(N\geq N_n\mbox{ or }s\geq s_n)\leq \exp\{-c_1n\varepsilon_n^{*2}\}\label{smallprior}\\
&\log \frac{m(\boldsymbol{\boldsymbol{D}^{(n)}})}{p_0(\boldsymbol{\boldsymbol{D}^{(n)}})}\geq \exp\{-c_2n\varepsilon_n^{*2}\},\quad \mbox{w.h.p.}\label{marg}
\end{align}
where $c_1$ increases to infinity as  $c_s$ increases, $c_2>0$ is an absolute constant, $m(\boldsymbol{D}^{(n)})=\int p_{\theta}(\boldsymbol{D}^{(n)})d\pi(\theta)$ is the marginal density.

Inequality ($\ref{smallprior}$) is true, since
\[
\begin{split}
&-\log\pi(N> N_n) \asymp N_n\log N_n\succ n\varepsilon_n^{*2} \mbox{ and }\\
&-\log\pi(s> s_n)\geq C\lambda_s s_n 
\asymp n\varepsilon_n^{*2},
\end{split}
\]
hold for some constant $C$.

To prove (\ref{marg}), it is suffice to find a subset $\mathcal F_s\subset\mathcal F$, such that $\pi(\mathcal F_s)\geq \exp\{-c_3n\varepsilon_n^{*2}\}$ and w.h.p. $p_\theta(\boldsymbol{D}^{(n)})/p_0(\boldsymbol{D}^{(n)})\\ \geq\exp\{-c_4n\varepsilon_n^{*2}\}$ for any $p_\theta\in\mathcal F_s$. Such $\mathcal F_s$ can be defined as $\{f_{\theta} \in \mathcal{F}(L, \boldsymbol{p}^*=(12pN^*,\dots,12pN^*)',s^{\ast}):\|f_{\theta} - f_0\|_{\infty}\leq \varepsilon_n^*\}$,

First, we show that $p_{\theta}(\boldsymbol{D}^{(n)})/p_0(\boldsymbol{D}^{(n)})\geq \exp\{-c_4n\varepsilon_n^{*2}\}$ for any $p_{\theta} \in \mathcal{F}_s$. Note that
\[
\begin{split}
&-\log p_{\theta}(\boldsymbol{D}^{(n)})/p_0(\boldsymbol{D}^{(n)})\\
= &-\frac{1}{2\sigma^2_{\epsilon}}\sum^n_{i=1}[(Y_i - f_{0}(X_i))^2 - (Y_i - f_{\theta}(X_i))^2] \\
\leq &\frac{1}{2\sigma^2_{\epsilon}}[n\|f_{\theta} - f_0\|^2_{\infty} + 2|\langle \boldsymbol{Y}^{(n)}-f_0(\boldsymbol{X}^{(n)}),\\ 
&\qquad \qquad \qquad \qquad \qquad f_\theta(\boldsymbol{X}^{(n)})-f_0(\boldsymbol{X}^{(n)})\rangle |].
\end{split}
\]
Note that $\boldsymbol{Y}^{(n)}-f_0(\boldsymbol{X}^{(n)})$ is a vector of i.i.d. normal $N(0,\sigma_\epsilon^2)$, then by concentration inequality,
w.h.p, 
\[
|\langle \boldsymbol{X}^{(n)}-f_0(\boldsymbol{X}^{(n)}), f_\theta(\boldsymbol{X}^{(n)})-f_0(\boldsymbol{X}^{(n)})\rangle |\leq cn\varepsilon_n^{*2}
\]
for some $c>0$, and we can conclude that w.h.p.,
\[
\frac{p_{\theta}(\boldsymbol{D}^{(n)})}{p_0(\boldsymbol{D}^{(n)})}\geq \exp\{-c_4n\varepsilon_n^{*2}\}
\]

Second, we prove that $\pi(\mathcal{F}_s) \geq \exp \{-c_3n\varepsilon^{*2}_n\}$ in the following.
By condition 5.2, $\xi_n^*\asymp r_n^* = o(\varepsilon_n^{*2}) $, hence there must exists a NN $\widehat{f}_{\widehat{\theta}} \in \mathcal{F}(L, \boldsymbol{p}^* s^{\ast},\widehat{\gamma})$, where $\widehat{\gamma}$ denotes a specific pattern of nonzero links among $\widehat{\theta}$, s.t.
\[
\|\hat{f}_{\hat{\theta}} - f_0\|_\infty \lesssim \varepsilon_n^*/2.
\]

By triangle inequality, 
\[
\begin{split}
&\{f_{\theta} \in \mathcal{F}(L,\boldsymbol{p}^*,s^*):\|f_{\theta} - f_0\|_\infty\leq \varepsilon_n^*\}\\
\supset &\{f_\theta \in \mathcal{F}(L,\boldsymbol{p}^*,s^*,\widehat{\gamma}):\|f_{\theta} - \widehat{f}_{\widehat{\theta}}\|_{\infty}\leq \frac{\varepsilon_n^*}{2}\}.
\end{split}
\]
Furthermore, from the proof of Lemma 10 of \cite{Schmidt-Hieber2017Nonparametric}, we have
\[
\begin{split}
&\{f_{\theta}\ \in \mathcal{F}(L,\boldsymbol{p}^*,s^*,\widehat{\gamma}): \|f_{\theta} - \hat{f}_{\hat{\theta}}\|_{\infty} \leq \frac{\varepsilon_n^*}{2}\} \\
\supset &\{f_\theta: \|\theta\|_{\infty}\leq 1 \mbox{ and } \|\theta - \hat{\theta}\|_{\infty}\leq \frac{\varepsilon_n^*}{2VL} \},
\end{split}
\]
where $V=L(12pN^*+1)$.

Therefore,  
\[
\begin{split}
&\pi\{f_{\theta} \in \mathcal{F}(L,\boldsymbol{p}^*,s^*):\|f_{\theta} - f_0\|_\infty\leq \varepsilon_n^*\} \\
> &\frac{\pi\{f_\theta \in \mathcal{F}(L,\boldsymbol{p}^*,s^*,\widehat{\gamma}):\|f_{\theta} - \widehat{f}_{\widehat{\theta}}\|_{\infty}\leq \frac{\varepsilon_n^*}{2}\}}{{T \choose s^*}}\\
>& e^{-Ls^{\ast}\log(12pN^{\ast})}\pi\{\theta: \|\theta\|_{\infty} \leq 1 \mbox{ and } \\
&\|\theta - \hat{\theta}\|_{\infty} \leq \frac{\varepsilon_n^*}{2VL}\},
\end{split}
\]
where $T$ denotes the total number of edge in network $\mathcal{F}(L,\boldsymbol{p}^*,s^*)$.
Note that 
\[
\begin{split}
&\pi\{\theta: \|\theta\|_{\infty} \leq 1 \mbox{ and } \|\theta - \hat{\theta}\|_{\infty} \leq \frac{\varepsilon_n^*}{2VL})\}\\
\approx &\exp\{-s^{\ast}\log(\frac{2VL}{\varepsilon_n^*})\}.
\end{split}
\]

Therefore, it is sufficient to show that
\[
\begin{split}
&Ls^{\ast}\log(12pN^{\ast}) + s^{\ast}\log(\frac{2L^2(12pN^{\ast}+1)}{\varepsilon_n^*})\\
\leq &c_3n\varepsilon^{*2}_n,
\end{split}
\]
which hold trivially due to the definition of $\varepsilon_n^*$.

\end{proof}

{\noindent \bf Proof of Theorem 5.1}
\begin{proof}
Denote $\delta_{\widehat N}$ and $\delta_{\widehat s}$ be the degenerate VB posterior of $N$ and $s$. We claim that with dominating probability,
\begin{align}\label{complexity}
\widehat N<N_n \mbox{ and } \widehat s<s_n.
\end{align}
Therefore, it will be equivalent to consider the truncated prior 
$\widetilde \pi(N)\propto \pi(N)1(N<N_n)$ and $\widetilde\pi(s)\propto \pi(s)1(s<s_n)$.

Note that
\[
\begin{split}
&-\log \pi(N=N^*) \leq -\log \widetilde\pi(N=N^*)\\
\leq& \lambda + \log N^*! - N^*\log \lambda
\asymp N^*\log N^*\\
\leq &s^*\log N^* =O(n\varepsilon_n^{*2} ),
\end{split}
\]
and 
\[
\begin{split}
&-\log \pi(s=s^*) =O(\lambda_s s^*)
=O(n\varepsilon_n^{*2} ).
\end{split}
\]
Therefore, the conditions of Lemma \ref{truncated} hold and we conclude the proof.

Recall $q^*(\theta, N, s)\in\mathcal{Q}_{N,s}$ which is defined in the proof of Lemma \ref{truncated}, and we prove (\ref{complexity}) by showing that w.h.p.,
\begin{align}
\mbox{KL}(q^*(\theta, N, s)\|&\pi(\theta,N,s|\boldsymbol{D}^{(n)})) \nonumber\\
&\leq \mbox{KL}(q(\theta, N, s) \|\pi(\theta,N,s|\boldsymbol{D}^{(n)})),\label{KLL}
\end{align}
for any $q\in\mathcal{Q}_{N,s}$ whose marginal degenerate distribution of $N$ is large than $N_n$ or marginal degenerate distribution of $s$ is greater than $s_n$. Note that
\[
\begin{split}
   & \frac{1}{n}\mbox{KL}(q^*(\theta, N, s)\|\pi(\theta,N,s|\boldsymbol{D}^{(n)}))\\
   =&\frac{1}{n}\mbox{KL}(q^*(\theta, N, s)\|\pi(\theta, N, s))+\frac{1}{n}\mathbb{E}_{q^*}\log\frac{p_0(\boldsymbol{D}^{(n)})}{p_\theta(\boldsymbol{D}^{(n)})}\\
   + &\frac{1}{n}\log\frac{m(\boldsymbol{D}^{(n)})}{p_0(\boldsymbol{D}^{(n)})}.
\end{split}
\]
The sum of the first two terms in above equation, as shown in the proof of Lemma \ref{truncated}, is $O(\varepsilon_n^{*2}+r_n^*)=O(\varepsilon_n^{*2})$. 
For the third term, by LLN, it converges to constant $-\mbox{KL}(P_0\|m)\leq 0$.

Due to Lemma \ref{prior3}, $\mbox{KL}(q(\theta, N, s) \|\pi(\theta,N,s|\boldsymbol{D}^{(n)}))\geq \\ c_0n\varepsilon_n^{*2}$, and the constant $c_0$ increases to infinity as $c_s$ increases. Therefore, providing a sufficiently large $c_s$, (\ref{KLL}) holds.

\end{proof}

\subsection{Remarks for proofs of Corollaries  6.1-6.4.}
The proofs for Corollaries 6.1 and 6.3 are straightforward, and they are directly implied by Theorem 4.1.

For the proofs of Corollaries 6.2 and 6.4, we comment that Theorem 5.1 actually holds for any $(N^*,s^*)$ which satisfies Conditions 
5.1, 5.3 and $\xi_n^*=O(r_n^*)$, but is not necessarily the exact minimization of $r_n^*+\xi_n^*$. Therefore, in this case we can still use Theorem 5.1 to prove Corollaries 6.2 and 6.4.

\subsection{Proof of Theorem 7.1}
\begin{proof}



For any $M_n \rightarrow \infty$, there always exists some $\widetilde M_n$ satistfying that $1 \prec \widetilde M_n = O(M_n)$ and $\gamma_n \widetilde M_n \widetilde \varepsilon^2_n = o(1)$.

Then, for any $\theta\in \mathcal{G} \cap \{\theta: L_2^2(f_0,f_\theta) \geq\widetilde M_n\widetilde \varepsilon_n^2\}$,
\begin{align}
  &  d^2(P_\theta, P_0) \geq \int_S(1-\exp\{-(f_\theta(x)-f_0(x))^2/8\sigma_\epsilon^2\}) dP(x) \nonumber\\
  \geq &   \frac{(1-\exp\{- \gamma_n L^2_2(f_0,f_\theta)/8\sigma_\epsilon^2\}) }{ \gamma_n L^2_2(f_0,f_\theta)}\int_S(f_\theta(x)-f_0(x))^2 dP(x)\nonumber\\
  \geq& \frac{(1-\exp\{- \gamma_n L^2_2(f_0,f_\theta)/8\sigma_\epsilon^2\}) }{ \gamma_n }\kappa\nonumber\\
  \geq & \frac{(1-\exp\{- \gamma_n \widetilde M_n\widetilde \varepsilon_n^2/8\sigma_\epsilon^2\}) }{ \gamma_n }\kappa
\geq c_M \widetilde M_n\widetilde \varepsilon_n^2, \label{eq:gen}
\end{align}
for some constant $c_M >0$, where the second inequality holds since $|f_{\theta}(X) - f_0(X)|^2$ is upper bounded by $\gamma_nL^2_2(f_0, f_{\theta})$ on $\mathcal S$, and the last inequality is due to the fact that $\gamma_n\widetilde M_n\widetilde \varepsilon_n^2=o(1)$. (\ref{eq:gen}) implies
\begin{equation} \label{eq:gen2}
\mathcal{G} \cap \{ L_2^2(f_0,f_\theta) \geq \widetilde  M_n\widetilde \varepsilon_n^2\}\subset\{{d^2(P_\theta, P_0)\geq c_M \widetilde  M_n\widetilde \varepsilon_n^2}\}.
\end{equation}

By Theorem 4.1, w.h.p.,
\[
\int d^2(P_\theta, P_0) \widehat q(\theta) = O(\widetilde \varepsilon_n^2),
\]
which implies that 
\[
\int_{d^2(P_\theta, P_0)\geq c_M\widetilde M_n\widetilde \varepsilon_n^2} \widehat q(\theta) = O(1/\widetilde M_n)=o(1).
\]

Combined with (\ref{eq:gen2})
\[
\begin{split}
    &\int_{\mathcal{G} \cap \{ L_2^2(f_0,f_\theta)> M_n\widetilde \varepsilon_n^2\}} \widehat q(\theta) \leq \int_{\mathcal{G} \cap \{ L_2^2(f_0,f_\theta)> \widetilde M_n\widetilde \varepsilon_n^2\}} \widehat q(\theta) \\
    \leq &\int_{d^2(P_\theta, P_0)\geq c_M\widetilde M_n\widetilde \varepsilon_n^2} \widehat q(\theta) = O(1/\widetilde M_n)=o(1), w.h.p.
\end{split}
\]

\end{proof}



\section{Implementation}
In this section, the implementation details of ASVI are provided.

\subsection{Approximated negative ELBO}
The exact AVSI algorithm requires one to figure out $\Omega(N,s)$ and compare $\Omega(N,s)$ across different choices of $N$ and $s$. Our approximation integrates out the sparsity variable $s$ in the hierarchical modeling, i.e., we consider the prior 
\begin{equation} \label{adaprior}
\begin{split}
&\pi(N) = \frac{\lambda^N}{(e^{\lambda}-1)N!}, \mbox{ for some } N\in\mathbb{Z^+}, \\
&\pi(\gamma|N) = c_1e^{-\lambda_s\Gamma}/{H \choose \Gamma}, \mbox{ with } \Gamma=\sum^H_{i=1}\gamma_i, \mbox{ for } c_1>0, \\
&\pi(\theta_i|\gamma_i) = \gamma_i \mathcal M_0(\theta_i) + (1 - \gamma_i)\delta_0,
\end{split}
\end{equation}
where $H$ is the total number of possible connections given width multiplier $N$.
The corresponding VB family is 
\[\begin{split}
   & q(N)=\delta_{\widebar N}, \quad q(\gamma_i|N) = \mbox{Bern}(\nu_i),\\
   &q(\theta_i|\gamma_i) = \gamma_i\mathcal{M}(\theta_i)+(1-\gamma_i)\delta_{0},
\end{split}
\]
for some $\widebar N\in\mathbb{Z^+}$.

Under Gaussian slab distribution, the negative ELBO (up to a constant) corresponding to the above VB modeling is a function of $\widebar N, \mu_i, \sigma_i$ and $\nu_i$'s,
\begin{align} 
    -\Omega = &- \int \log p(\boldsymbol{D}^{(n)}|\theta, \gamma)q(\theta|\gamma)q(\gamma|\widebar N)d\theta d\gamma  \nonumber\\
    & + \sum^H_{i=1}q(\gamma_i=1)\mbox{KL}(\mathcal{N}(\theta_i;\mu_i,\sigma^2_i)\|\mathcal{N}(\theta_i;0,\sigma^2_0)) \nonumber\\
    & + \mbox{KL}(q(\gamma|\widebar N)\|\pi(\gamma|\widebar N)) - \log\pi(\widebar N). \nonumber
\end{align}

Let 
\[
\begin{split}
\mathcal L = &-\int \log p(\boldsymbol{D}^{(n)}|\theta, \gamma)q(\theta|\gamma)q(\gamma|\widebar N)d\theta d\gamma \nonumber \\ 
    & + \sum^H_{i=1}q(\gamma_i=1)\mbox{KL}(\mathcal{N}(\theta_i;\mu_i,\sigma^2_i)\|\mathcal{N}(\theta_i;0,\sigma^2_0))\nonumber \\ 
    & + \mbox{KL}(q(\gamma|\widebar N)\|\pi(\gamma | \widebar N))\\
    := &\mathcal L_1 + \mathcal L_2 + \mathcal L_3,
\end{split}
\]
and 
\begin{align} 
&-\Omega(\widebar N)=\underset{\{\mu_i, \sigma_i,\nu_i\}}{\arg\min} \mathcal L. \label{el}
\end{align}
Thus the optimal $N$ value $\widehat N$ maximizes the penalized ELBO: $\Omega_p(\widebar N)=\Omega(\widebar N)+\log\pi(\widebar N) $.

To approximate and optimize $\mathcal L$, we study each of the three terms:

i) $\mathcal L_1= - \int \log p(\boldsymbol{D}^{(n)}|\theta, \gamma)q(\theta|\gamma)q(\gamma|\widebar N)d\theta d\gamma$ requires Monte Carlo estimation. We use reparameterization trick \cite{Kingma2014VAE, Rezende2014Stochastic} for the normal slab distribution $\mathcal{M}(\theta)$, i.e., $\mathcal{M}(\theta_i)$ is equivalent in distribution to $\mu_i+\sigma_i\epsilon_i$  for $\epsilon_i \sim \mathcal N(0,1)$. Gumbel-softmax approximation \cite{Maddison17Concrete, Jang17Categorical} is employed for the binary variable $\gamma$, that is
\[\begin{split}
    &\gamma_i = 1(G_\tau(\nu_i;u_i) > 0.5),\\
    &G_\tau(\nu_i;u_i)=\frac{1}{1+\exp(-(\log \frac{\nu_i}{1-\nu_i} + \log \frac{u_i}{1-u_i})/\tau)} 
\end{split}
\]
for $u_i \sim \mathcal U(0,1)$, where $\tau$ is called the temperature and is set as 0.5 in our implementation. In back-propagation, $\gamma_i$ is used in the forward pass and $G_\tau(\nu_i;u_i)$ is used in the backward pass to compute the gradient. 
In other words, let $g(\mu_i,\sigma_i,\nu_i;\epsilon_i,u_i) = 1(G_\tau(\nu_i;u_i)>0.5)(\mu_i+\sigma_i\epsilon_i)$ and $g'(\mu_i,\sigma_i,\nu_i;\epsilon_i,u_i) = G_\tau(\nu_i;u_i)(\mu_i+\sigma_i\epsilon_i)$, then the stochastic estimator \cite{Kingma2014VAE} for $\mathcal{L}_1$ (used for forward pass) is
\begin{equation} \label{loss1}
\begin{split}
\widetilde{\mathcal{L}_1} = -\frac{n}{m}\frac{1}{K}\sum^m_{j=1}&\sum^K_{k=1}\log p(D_j|\theta^{(k)}),\\
\end{split}
\end{equation}
where $\theta^{(k)} = (\theta^{(k)}_1,\dots,\theta^{(k)}_H)',
\theta^{(k)}_i = g(\mu_i,\sigma_i,\nu_i;\epsilon_i^{(k)},u_i^{(k)})$.
$D_j$'s are randomly drawn from $D$, $\epsilon_i^{(k)}$'s and $u_i^{(k)}$'s are randomly drawn from $\mathcal N (0,1)$ and $\mathcal U (0,1)$ respectively, $n$ is the sample size, $m$ is the minibatch size and $K$ is the Monte Carlo sample size. The stochastic estimator  for $\nabla\mathcal{L}_1$ (used for backward pass) is
\begin{equation}\label{grad}
\begin{split}
&\widetilde\nabla_{\mu_i} \mathcal{L}_1 
=-\frac{n}{m}\frac{1}{K}\sum^m_{i=1}\sum^K_{k=1}\nabla_{\mu_i} \log p(D_i|\widetilde\theta^{(k)}),\\
&\widetilde\nabla_{\sigma_i} \mathcal{L}_1 
=-\frac{n}{m}\frac{1}{K}\sum^m_{i=1}\sum^K_{k=1}\nabla_{\sigma_i} \log p(D_i|\widetilde\theta^{(k)}),\\
&\widetilde\nabla_{\nu_i} \mathcal{L}_1 
=-\frac{n}{m}\frac{1}{K}\sum^m_{i=1}\sum^K_{k=1}\nabla_{\nu_i} \log p(D_i|\widetilde\theta^{(k)}).
\end{split}
\end{equation}
where $\widetilde\theta^{(k)} = (\widetilde\theta^{(k)}_1,\dots,\widetilde\theta^{(k)}_H)'$, $\widetilde\theta^{(k)}_i = g'(\mu_i,\sigma_i,\nu_i;\epsilon_i^{(k)},u_i^{(k)})$.


ii) $\mathcal L_2$ is straightforward that
\begin{align} 
&\sum^H_{i=1}q(\gamma_i=1)\mbox{KL}(\mathcal{N}(\theta_i;\mu_i,\sigma^2_i)\|\mathcal{N}(\theta_i;0,\sigma^2_0)) \nonumber\\
= &\sum^H_{i=1} \nu_i (\log \frac{\sigma_{0}}{\sigma_i} + \frac{\sigma^2_i + \mu^2_i}{2\sigma^2_{0}}- 0.5). \label{loss2}
\end{align}

iii) To compute $\mathcal L_3$, certain approximation is needed.
Denote $\Gamma^H$ as the set of all possible $\gamma = (\gamma_1, \ldots, \gamma_H)$, then
\[
\begin{split}
&\mbox{KL}(q(\gamma|\widebar N)\|\pi(\gamma|\widebar N)) \\
= &\sum_{\gamma \in \Gamma^H} \log \frac{q(\gamma_1, \ldots, \gamma_H)}{\pi(\gamma_1, \ldots, \gamma_H)} q(\gamma_1, \ldots, \gamma_H) \\
= &\sum^H_{t=0} \sum_{\Gamma=t}\log \frac{q(\gamma_1, \ldots, \gamma_H)}{\pi(\gamma_1, \ldots, \gamma_H)} q(\gamma_1, \ldots, \gamma_H)
\end{split}
\]
For the sake of fast computation, we approximate the VB distribution $q(\gamma)$ by iid Bernoulli distribution $q(\gamma)\approx \prod \widetilde{\nu}^{\gamma_i}(1-\widetilde{\nu})^{1-\gamma_i}$, where $\widetilde{\nu} = \frac{1}{H}\sum^H_{i=1}\nu_i$. Under this approximation:
\[
\begin{split}
&\sum_{\Gamma=t}\log \frac{q(\gamma_1, \ldots, \gamma_H)}{\pi(\gamma_1, \ldots, \gamma_H)}q(\gamma_1, \ldots, \gamma_H) \\
\approx&\sum_{\Gamma=t}\log \frac{\widetilde{\nu}^t (1-\widetilde{\nu})^{H-t}}{\pi(\gamma|\Gamma=t)}\widetilde{\nu}^t (1-\widetilde{\nu})^{H-t} \\
=&{H \choose t} \log \frac{\widetilde{\nu}^t (1-\widetilde{\nu})^{H-t}}{\pi(\gamma|\Gamma=t)}\widetilde{\nu}^t (1-\widetilde{\nu})^{H-t}\\
=&\log \frac{{H \choose t}\widetilde{\nu}^t (1-\widetilde{\nu})^{H-t}}{{H \choose t}\pi(\gamma|\Gamma=t)}{H \choose t}\widetilde{\nu}^t (1-\widetilde{\nu})^{H-t}\\
=&\log Pr(\Gamma=t) Pr(\Gamma=t) + \lambda_s t Pr(\Gamma=t) + C_1 \\
\end{split}
\]
where $C_1$ is some constant. Therefore, $\mbox{KL}(q(\gamma)\|\pi(\gamma)))$ is approximated by
\begin{align} 
&\sum_{\gamma \in \Gamma^H} \log \frac{q(\gamma_1, \ldots, \gamma_H)}{\pi(\gamma_1, \ldots, \gamma_H)} q(\gamma_1, \ldots, \gamma_H) \nonumber\\
= &\sum^H_{t=0} \sum_{\Gamma=t}\log \frac{q(\gamma_1, \ldots, \gamma_H)}{\pi(\gamma_1, \ldots, \gamma_H)} q(\gamma_1, \ldots, \gamma_H) \nonumber\\
=&\sum^H_{t=0} \log Pr(\Gamma=t) P(\Gamma=t) + \lambda_s \sum^H_{t=0}t Pr(\Gamma=t) + C_2  \nonumber\\
=& -\mathbb{H}(\Gamma) + \lambda_s \mathbb{E}(\Gamma) + C_2 \nonumber\\
\approx & -0.5\log_2(2\pi e \sum \nu_i(H- \sum \nu_i)/H) + \lambda_s \sum^H_{i=1}\nu_i + C_2 \nonumber\\
:= &\widetilde{\mathcal L_3}\label{loss3}
\end{align}
where $\mathbb{H}(\Gamma)$ is the entropy of $\Gamma$ and $C_2$ is some constant.

\subsection{Algorithm}

\begin{algorithm}
\caption{Adaptive sparse variational inference with normal slab distribution.} \label{alg:vb}
\begin{algorithmic}[1]
\State {Hyperparameters: $\lambda$, $\lambda_s$, $\sigma_0$}
\State {Parameters: $\mu, \sigma', \nu'$} 
\State {Candidate set of $\widebar N$: $N_A$}

\FORALLP {$\widebar N \in N_A $}
    \Repeat
        \State {$\{D_j\}^m_{j=1}$ $\gets$ Sample a minibatch of size $m$ }
        \State {$\{\epsilon^{(k)}_i\}_{1\leq k\leq K, 1\leq i\leq H}$ $\gets$ iid samples from $\mathcal N (0,1)$} 
        \State {$\{u^{(k)}_i\}_{1\leq k\leq K, 1\leq i\leq H}$ $\gets$ iid samples from $\mathcal U(0, 1)$ } 
        \State {$\widetilde{\mathcal L}$ $\gets$ (\ref{loss1}), (\ref{loss2}) and (\ref{loss3})}
        \State {$\widetilde{\nabla}_{\mu_i}\mathcal L$, $\widetilde{\nabla}_{\sigma_i}\mathcal L$, $\widetilde{\nabla}_{\nu_i}\mathcal L$ $\gets$ Gradients of $\mathcal L_2$ and $\widetilde{\mathcal L_3} $ \\
        \qquad \qquad \qquad \qquad \qquad \qquad \quad together with (\ref{grad})}
        \State {$\widetilde{\nabla}_{\sigma'_i}\mathcal L$ $\gets$ $\widetilde{\nabla}_{\sigma_i}\mathcal L \cdot \nabla_{\sigma'_i}\sigma_i$}
        \State {$\widetilde{\nabla}_{\nu'_i}\mathcal L$ $\gets$ $\widetilde{\nabla}_{\nu_i}\mathcal L \cdot \nabla_{\nu'_i}\nu_i$}
        \State {$\mu_i, \sigma'_i, \nu'_i$ $\gets$ Update with $\widetilde{\nabla}_{\mu_i}\mathcal L$, $\widetilde{\nabla}_{\sigma'_i}\mathcal L$, $\widetilde{\nabla}_{\nu'_i}\mathcal L$\\
        \qquad \qquad \qquad \qquad  using gradient descent algorithms\\
        \qquad \qquad \qquad \qquad  (e.g. RMSprop or Adam)}
    \Until{convergence of $\widetilde{\mathcal L}$}
    \State {$-\widetilde{\Omega}(\widebar N)$ $\gets$ $\widetilde{\mathcal L}$}
    \State {$-\widetilde\Omega_p(\widebar N)$$\gets$ $-\widetilde{\Omega}(\widebar N)-\log \pi(\widebar N)$ with ($\widebar N$, $\lambda$)}
\ENDFAP
\State {$\widehat{N}$ = $\arg \min_{\widebar N \in N_{A} }(-\widetilde\Omega_p(\widebar N))$}
\State \Return {$\widehat{N}$ and $(\mu, \sigma', \nu'|\widehat{N})$}
\end{algorithmic}
\end{algorithm}

An additional re-parametrization transformation for $\sigma$ and $\nu$ is used,
\[
    \sigma'_i = \log (\exp(\sigma_i)-1), \: \nu'_i = \log (\frac{1 - \nu_i}{\nu_i}),
\]
such that $\sigma_i'$ and $\nu_i'\in \mathbb{R}$.
Let $\widetilde{\mathcal L}$ and $\widetilde{\nabla} \mathcal L$
denote the working approximations of $\mathcal L$ and $\nabla \mathcal L$, then $\widetilde{\mathcal L} = \widetilde{\mathcal L_1} + \mathcal L_2 + \widetilde{\mathcal L_3}$ using 
(\ref{loss1}), (\ref{loss2}) and (\ref{loss3}). Furthermore, there exist explicit gradients of $
\mathcal L_2$ and $\widetilde{\mathcal L_3}$ with respect to $\nu'_i$, $\mu_i$ and $\sigma'_i$, which facilitates the calculation of the approximate gradient $\widetilde{\nabla} \mathcal L$ along with (\ref{grad}).



The complete adaptive sparse variational inference is described in Algorithm \ref{alg:vb}, where we use $\widetilde \Omega(\widebar N)$ and $\widetilde \Omega_p(\widebar N)$ to denote the working approximations of $\Omega(\widebar N)$ and  $\Omega_p(\widebar N)$ respectively.


\subsection{Remaining implementation details}
\subsubsection{Teacher network}
The batch size is set as $m = 1024$, and Monte Carlo size $K=1$ during training. Adam is used for optimization with a learning rate of $5\times 10^{-3}$, and the number of epochs is 7000. $\lambda_s$ is chosen as 3 ($a=0.1$) and $\lambda$ is chosen as 10, $\sigma_0$ is fixed at 0.8. 

\subsubsection{UCI datasets}
For all the datasets, the batch size is set as $m = 256$, Monte Carlo size $K$ is set as 1 during training, and Adam is used for optimization with a learning rate of $1\times 10^{-3}$. The number of epochs is 1000 for "Naval", "Power Plant" and "Protein", 2000 for "Kin8nm" and 100 for "Year". $\sigma_0$ and $\sigma_{\epsilon}$ are determined by a grid search that yields the best prediction accuracy.

\subsubsection{MNIST}
The batch size is set as $m = 512$, and Monte Carlo size $K=1$ during training. RMSprop is used for optimization with a learning rate of $5\times 10^{-3}$, and the number of epochs is 300. $\lambda_s$ is chosen as 50 ($a = 1.5$) and $\lambda$ is chosen as 600,  $\sigma_0$ is fixed at 2. MNIST data is standardized by mean of 0.1307 and standard deviation of 0.3081.

\end{document}